\documentclass[11pt]{article}
\usepackage{amsfonts}
\usepackage{changepage}
\usepackage{amsmath,amssymb,amsthm,amsfonts,enumerate}

\usepackage{graphicx}
\usepackage{float}
\usepackage[scriptsize,nooneline]{subfigure}
\usepackage{indentfirst}
\usepackage{cite}
\usepackage{color}
\usepackage{mathrsfs}
\usepackage{epstopdf}
\usepackage{setspace}
\usepackage[labelfont=footnotesize,font=footnotesize]{caption}
\usepackage[colorlinks, citecolor=red]{hyperref}
\everymath{\displaystyle}
\newcommand{\R}{{\mathbb R}}
\newtheorem{Assumption}{Assumption}[section]
\newtheorem{theorem}{Theorem}[section]
\newtheorem{lemma}{Lemma}[section]
\newtheorem{definition}{Definition}[section]
\newtheorem{proposition}{Proposition}[section]
\newtheorem{remark}{Remark}[section]

\begin{document}

	\title{\begin{flushleft} \bf The local Morrey-type space Associated with Ball Quasi-Banach Function Spaces and Application\\[10pt] \rm{\Large Mingwei Shi and Jiang Zhou*} \end{flushleft}}
	\date{\it College of Mathematics and System Sciences, Xinjiang University, Urumqi 830046, China}
	\maketitle
	
	\begin{adjustwidth}{1cm}{1cm}
		\noindent{{\bf Abstract:}{ In this paper, we define for the first time the local Morrey-type space associated with ball quasi-Banach function spaces and show the related series of properties. In addition,  Hardy-Littlewood maximal operator's boundedness is proved.   
	We  investigate nonsmooth decomposition of the local Morrey-type space associated with ball quasi-Banach function spaces via the Hardy local Morrey-type spaces associated with ball quasi-Banach function spaces.
	And we consider  Hardy operator's boundedness. }}
	\end{adjustwidth}
	\par

	\baselineskip 15pt
	\section{Introduction}
	\renewcommand{\thefootnote}{}
	\footnotetext{*Corresponding author. The research was supported by National Natural Science Foundation of China (12061069).}
In 1938, Morrey \cite{CB1938} introduced the first study of Morrey spaces $\mathcal{M}_{p}^q(\R^n)$ for partial differential equations. In 1975, D. Adams \cite{DR1975} established that Morrey spaces $\mathcal{L}_{p,\lambda}(\R^n)$ can describe the boundedness property of Riesz potentials. In 2009, Samko \cite{Sn2009} examined the boundedness of  Hardy operator on Morrey spaces $\mathcal{L}_{p,\lambda}(\R^n)$. In  2014, Iida and Sawano \cite{TY2014} studied atomic decomposition for Morrey spaces $\mathcal{M}_{p}^q(\R^n)$. 
\par   In 2004, Burenkov and Guliyev \cite{BV2004} proposed local Morrey-type spaces $LM_{p \theta,w}(\R^n)$, a sufficient and necessary condition for the boundedness of Hardy-littlewood maximal operator was shown on spaces $LM_{p \theta,w}(\R^n)$. In 2010, Burenkov and Nursultanov \cite{BN2010}  description of interpolation for local Morrey spaces $LM_{p, q}^{\lambda}(\R^n)$. In 2011, Burenkov et.al. \cite{BJ2011} studied the boundedness of the Hardy operator on spaces  $LM_{p \theta,w}(\R^n)$. And in 2014,  Batbold and Sawano \cite{BS2014} obtained the decomposition of  spaces $LM_{p, q}^{\lambda}(\R^n)$. In 2017, Guliyev \cite{GH2017} researched the decomposition of  spaces $LM_{p \theta,w}(\R^n)$.
Other related references   \cite{A1,A2,A3,A4,A5,A6}.

The variable exponential Lebesgue spaces $L_{p(\cdot)}(\R)$  first appeared  in 1931 by Orlicz \cite{Wo1931}. In 2008, Kokilashvili and  Meskhi \cite{VA2008} introduced Variable Morrey
Spaces $M_{q(\cdot)}^{p(\cdot)}(X)$  In 2020, the local variable exponential  Morrey space $LM_{u}^{p(\cdot)}(\R^n)$ was proposed by Yee et.al.  \cite{YC2020}. Other related references \cite{ HK2291,HP221}.

Benedek and Panzone \cite{AB1961} proposed mixed Lebesgue spaces $L_{\vec{p}}(\R^n)$ in 1961. Mixed Morrey spaces $M_{\vec{q}}^{p}(\mathbb{R}^n)$ were  introduced by Nogayama \cite{NT2019} in 2019. In 2021, Zhang and Zhou \cite{ZH2021} proposed local mixed Morrey spaces $LM_{\vec{p} \theta,w}(\R^n)$. In 2022,  Shi and Zhou \cite{SJ2022} obtained a sufficient and necessary condition for the Hardy-littlewood maximal operator on spaces $LM_{\vec{p} \theta,w}(\R^n)$. In 2022, Shi and Zhou \cite{SD2022} obtained nonsoomth decompositions of spaces $LM_{\vec{p} \theta,w}(\R^n)$. Other related references \cite{ WM2022}.

In 2017, Sawano et.al. \cite{SS2017}  established Hardy spaces for ball quasi-Banach function spaces $H_X(\R^n)$. In 2019, Zhang et.al. \cite{Z2021} and Wang et.al. \cite{W2021}  established the weak Hardy-type space associated with ball quasi-Banach function space $WH_X(\R^n)$. In 2019, Ho  \cite{HA2019}  proposed Morrey-Banach spaces $M_X^u(\R^n)$. In 2019, Ho  \cite{HK2019} investigated Weak Type Estimates of Singular Integral Operators on spaces $M_X^u(\R^n)$. In 2021, Zhang et.al. \cite{ZYD202} established  Ball Campanato-Type function space $\mathcal{L}_{X,q,d,s}(\R^n)$. In 2022, Ho.\cite{HP2022} obtained the boundedness of some type of  maximal  operators on spaces $M_X^u(\R^n)$.  In 2022, Wang and Zhou \cite{WZ2022} by means of ball quasi-Banach spaces, it is proved that the Calderón-Zygmund singular integral operator is bounded on the generalized Orlicz space $L^{\varphi(\cdot)}(\R^n)$ via a new atomic decomposition. Currently, more and more researchers are solving problems with the help of ball quasi-Banach function spaces. Other more  related references  \cite{H1,H2,H3,H4,H5}.

In this paper, we define the local Morrey-type space associated with ball quasi-Banach function spaces. In Section 2,  some properties of the local Morrey-type space associated with ball quasi-Banach function spaces are derived. And the relationship of local Morrey-type spaces associated with ball quasi-Banach function spaces  to some of the other function  spaces is discussed. In Section 3, the boundedness of Hardy-littlewood maximal operators is obtained on the local Morrey-type space associated with ball quasi-Banach function spaces. In Section 4, an interpolation theorem is proved. In Section 5, vector valued maximal inequalities are obtained on the local Morrey-type space associated with ball quasi-Banach function spaces. In Section 6, predual spaces of  the local Morrey-type space associated with ball quasi-Banach function spaces  are proved. In section 7, the Hardy  local Morrey-type space associated with ball quasi-Banach function spaces are characterized. In section 8, we attain nonsmooth decompositions on local Morrey-type spaces associated with ball quasi-Banach function spaces. In Section 9, the boundedness of Hardy operator is obtained.
\section{Definition and properties}
\begin{definition}\cite{CB1988}
A quasi-Banach space $X \subset \mathcal{M}$ is called a ball quasi-Banach function space if it satisfies

(i) $\|f\|_X=0$ implies that $f=0$ almost everywhere;

(ii) $|g| \leq|f|$ almost everywhere implies that $\|g\|_X \leq\|f\|_X$;

(iii) $0 \leq f_m \uparrow f$ almost everywhere implies that $\left\|f_m\right\|_X \uparrow\|f\|_X$;

(iv) $B \in \mathbb{B}$ implies that $\chi_B \in X$, where $\mathbb{B}:=\{B(x,r): x \in \R^n \text{ and} ~r \in (0, \infty)\}$.

Moreover, a ball quasi-Banach function space $X$ is called a ball Banach function space if the norm of $X$ satisfies the triangle inequality: for all $f, g \in X$,
$$
\|f+g\|_X \leq\|f\|_X+\|g\|_X,
$$
and, for any $B \in \mathbb{B}$, there exists a positive constant $C_{(B)}$, depending on $B$, such that, for all $f \in X$,
$$
\int_B|f(x)| d x \leq C_{(B)}\|f\|_X
$$
\end{definition}

\begin{definition}\cite{CB1988}
	For any ball quasi-Banach function space $X$, the associate space (Köthe dual) $X^{\prime}$ is defined by
	$$
	X^{\prime}:=\left\{f \in \mathcal{M}:\|f\|_{X^{\prime}}:=\sup \left\{\|f g\|_{L^1\left(\mathbb{R}^n\right)}: g \in X,\|g\|_X=1\right\}<\infty\right\}
	$$
\end{definition}
\begin{definition}
	($\mathcal{F}_{N}, \mathcal{M} f(x)$) \cite{GH2017}
	$(i)$ The topology spaces of   $\mathcal{S}\left(\mathbb{R}^{n}\right) $ is defined by the norms  $\left\{\rho_{N}\right\}_{N \in \mathbb{N}} $,
	where $
	\rho_{N}(\varphi) := \sum_{|\alpha| \leq N} \sup _{x \in \mathbb{R}^{n}}(1+|x|)^{N}\left|\partial^{\alpha} \varphi(x)\right| \quad\left(\varphi \in \mathcal{S}\left(\mathbb{R}^{n}\right)\right) .
	$\par
	$$ \mathcal{F}_{N} := \left\{\varphi \in \mathcal{S}\left(\mathbb{R}^{n}\right): \rho_{N}(\varphi) \leq 1\right\}  ~~~\quad  N \in  \mathbb{N}\cup {\{0\}}. $$
	$(ii)$The space $ \mathcal{S}^{\prime}\left(\mathbb{R}^{n}\right) $ is the topological dual of  $\mathcal{S}\left(\mathbb{R}^{n}\right) $.
	\\(iii)Let  $f \in \mathcal{S}^{\prime}\left(\mathbb{R}^{n}\right) $, the grand maximal operator $ \mathcal{M} f$ of  f  is defined by
	$$
	\mathcal{M} f(x):=\mathcal{M}_{N} f(x) := \sup_{ t>0 ~,~ \varphi \in \mathcal{F}_{N}} \left\{\left|t^{-n} \varphi\left(t^{-1} \cdot\right) * f(x)\right|\right\}~~~~\text{for}~~ x \in \mathbb{R}^{n}. 
	$$

\end{definition}

\par 
 $C_{\text{c}}^{\infty}\left(\mathbb{R}^{n}\right)$ denotes the set of all compactly supported infinitely continously differentiable functions, the set of all polynomials of degree less than or equal to  d  is denoted by $ \mathcal{P}_{d}\left(\mathbb{R}^{n}\right) $.
\begin{lemma}\cite{SE1993}
	Let  $f \in \mathcal{S}^{\prime}\left(\mathbb{R}^{n}\right)\cap L^{\mathrm{loc}}_{1}\left(\mathbb{R}^{n}\right), d \in \mathbb{N}\cup {\{0\}}$  and  $j \in \mathbb{Z} $. Then there exist an index set  $K_{j}$ , collections of cubes $ \left\{Q_{j, k}\right\}_{k \in K_{j}}$  and functions  $\left\{\eta_{j, k}\right\}_{k \in K_{j}} \subset C_{\text {c }}^{\infty}\left(\mathbb{R}^{n}\right) $, which are all indexed by $ K_{j}$  for every  j, and a decomposition
	$$
	f=g_{j}+b_{j}, \quad b_{j}=\sum_{k \in K_{j}} b_{j, k},
	$$
	such that the following properties hold.
	\begin{itemize}
		\item[(\romannumeral1)]
		$  g_{j}, b_{j}, b_{j, k} \in \mathcal{S}^{\prime}\left(\mathbb{R}^{n}\right) $.
		\item[(\romannumeral2)]
		Define  $\mathcal{O}_{j} :=\left\{y \in \mathbb{R}^{n}: \mathcal{M} f(y)>2^{j}\right\}$  and consider its Whitney decomposition. Then the cubes  $\left\{Q_{j, k}\right\}_{k \in K_{j}}$  have the bounded intersection property, and
		\begin{equation}\label{88}
		\mathcal{O}_{j}=\bigcup_{k \in K_{j}} Q_{j, k}
		\end{equation}
		\item[(\romannumeral3)]
		Consider the partition of unity  $\left\{\eta_{j, k}\right\}_{k \in K_{j}} $ with respect to  $\left\{Q_{j, k}\right\}_{k \in K_{j}} $. Then each function $ \eta_{j, k}$  is supported in $ Q_{j, k}$  and
		$$
		\sum_{k \in K_{j}} \eta_{j, k}=\chi_{\left\{y \in \mathbb{R}^{n}: \mathcal{M} f(y)>2^{j}\right\}}, \quad 0 \leq \eta_{j, k} \leq 1 .
		$$
		\item[(\romannumeral4)] 
		$ g_{j}$  is an  $L^{\infty}\left(\mathbb{R}^{n}\right)$-function satisfying  $\left\|g_{j}\right\|_{L^{\infty}} \leq 2^{-j}.$
		\item[(\romannumeral5)]
		Each distribution $ b_{j, k}$  is given by  $b_{j, k}=\left(f-c_{j, k}\right) \eta_{j, k}$  with a certain polynomial $ c_{j, k} \in \mathcal{P}_{d}\left(\mathbb{R}^{n}\right) $ satisfying
		$$
		\left\langle f-c_{j, k}, \eta_{j, k} \cdot P\right\rangle=0 \text { for all } q \in \mathcal{P}_{d}\left(\mathbb{R}^{n}\right)
		$$
		and
		$$
		\mathcal{M} b_{j, k}(x) \lesssim \mathcal{M} f(x) \chi Q_{j, k}(x)+2^{j} \frac{\ell_{j, k}^{n+d+1}}{\left|x-x_{j, k}\right|^{n+d+1}} \chi_{\mathbb{R}^{n} \backslash Q_{j, k}}(x)
		$$
		$\text { for all } x \in \mathbb{R}^{n} \text {. }$
	\end{itemize}
	In the above, $ x_{j, k}$  and $ \ell_{j, k}$  denote the center and the edge-length of $ Q_{j, k} $.
\end{lemma}

\begin{definition} \cite{AC1961} (Lebesgue spaces)
	Let $ 0<p<\infty,$ the lebesgue space $L_p$ is defined to be the set of all measurable function f on $\R^n$ such that 
	$$
	\|f\|_{L_p(\R^n)}:=\left(\int_{\R^n} f(x)^p dx\right)^{\frac{1}{p}}
	$$
\end{definition}
\begin{definition}\cite{AB1961}
	(Mixed Lebesgue spaces) Let $\vec{p}=\left(p_1, \ldots, p_n\right) \in(0, \infty]^n$. Then define the mixed Lebesgue norm $\|\cdot\|_{L_{\vec{p}}}$ by
	$$
	\begin{aligned}
	\|f\|_{L_{\vec{p}}} := \left(\int_{\mathbb{R}} \cdots\left(\int_{\mathbb{R}}\left(\int_{\mathbb{R}}\left|f\left(x_1, x_2, \ldots, x_n\right)\right|^{p_1} \mathrm{~d} x_1\right)^{\frac{p_2}{p_1}} \mathrm{~d} x_2\right)^{\frac{p_3}{p_2}} \cdots \mathrm{d} x_n\right)^{\frac{1}{p_n}}
	\end{aligned}
	$$
\end{definition}

\begin{definition}\cite{Wo1931} (Variable Lebesgue spaces)
	Let $p(\cdot): \mathbb{R}^n \rightarrow[0, \infty)$ be a measurable function. Then the variable Lebesgue space $L_{p(\cdot)}\left(\mathbb{R}^n\right)$ is defined to be the set of all measurable functions $f$ on $\mathbb{R}^n$ such that
	$$
	\|f\|_{L_{p(\cdot)}\left(\mathbb{R}^n\right)}:=\inf \left\{\lambda \in(0, \infty): \int_{\mathbb{R}^n}[|f(x)| / \lambda]^{p(x)} d x \leq 1\right\}<\infty
	$$
\end{definition}

\begin{definition} (the local Morrey-type space associated with ball quasi-Banach function spaces)
	Let $\lambda \geq 0  $ and $ 0 <q \leq \infty .$ We denote the local Morrey space associated with ball quasi-Banach function spaces $LM_{X,q}^{\lambda}$, where X is the  ball quasi-Banach function space. For any functions $f \in L^{loc}_{1}$, we say $f\in LM^{\lambda}_{X,q}$ when the quasi-norms
$$
\|f\|_{L M_{X, q}^{\lambda}}=\left(\int_{0}^{\infty}\left(r^{-\lambda}\|f \chi_{B(0,r)}\|_{X}\right)^{q} \frac{d r}{r}\right)^{\frac{1}{q}}<\infty.
$$
\end{definition}
\begin{definition}(Morrey-Banach spaces) \cite{HA2019}
 Let $X$ be a rearrangement-invariant Banach function
 space and $u:(0, \infty) \rightarrow(0, \infty)$ be a Lebesgue measurable function. A measurable function f   belongs to the rearrangement-invariant Morrey spaces (Morrey-Banach spaces) $\mathcal{M}_X^u$ if it satisfies
$$
\|f\|_{\mathcal{M}_X^u}=\sup _{x_0 \in \mathbb{R}^n, r>0} \frac{1}{u(r)}\left\|\chi_{B\left(x_0, r\right)} f\right\|_X<\infty
$$
\end{definition}
\begin{remark} 
	Let  $q=\infty$  and   $\lambda \geq 0  $ , return to  the Morrey-Banach spaces.
$$
\|f\|_{M_{X}^{\lambda}}=\mathop{sup}\limits_{x \in \R^n} \| f(\cdot + x) \|_{LM_{X,\infty}^\lambda}=\mathop{sup}\limits_{x \in \R^n} \mathop{sup}\limits_{r>0}  r^{-\lambda} \|f \chi_{B(x,r)}\|_{X}
<\infty.
$$
\end{remark}

\begin{definition}(quasinorm derived from the  ball quasi-Banach function space)
For $f \in L^{loc}_{1}$, we consider the quasinorm where X is the  ball quasi-Banach function space.
$$
\begin{aligned}
\|f\|_{\widetilde{X}}&=|B(0,r)| \mathop{sup}\limits_{\rho \geq r} |B(0,\rho)|^{-1} \|f \chi_{B(0,\rho)}\|_{X}\\
&=r^n \mathop{sup}\limits_{\rho \geq r} \rho ^{-n} \|f \chi_{B(0,\rho)}\|_{X}.
\end{aligned}
$$
\end{definition}
\begin{definition}(quasinorm derived from the local Morrey-type space associated with ball quasi-Banach function spaces)
	Let   $\lambda \geq 0  $, $ 0 <q \leq \infty ,$ and $f \in L^{loc}_{1}$, we consider quasinorm where X is the  ball quasi-Banach function space.
	$$
	\|f\|_{\widetilde{L M}_{X, q}^{\lambda}}=\left(\int_{0}^{\infty}\left(r^{-\lambda} \|f \chi_{B(0,\rho)}\|_{\widetilde{X}}\right)^{q} \frac{d r}{r}\right)^{\frac{1}{q}}<\infty.
	$$
\end{definition}
\begin{definition}\cite{BS2014} (the heat kernel)
	Let $ t>0 $ and  $f \in \mathcal{S}^{\prime}\left(\mathbb{R}^{n}\right) $. The heat kernel is defined by  
	$$
	e^{t \Delta} f(x) :=\left\langle f, \frac{1}{\sqrt{(4 \pi t)^{n}}} \exp \left(-\frac{|x-\cdot|^{2}}{4 t}\right)\right\rangle \quad x \in \mathbb{R}^{n} .
	$$
\end{definition}
\begin{definition}
	(the Hardy local Morrey-type spaces associated with ball quasi-Banach function spaces) Let  $\lambda \geq 0  $, $ 0 <q \leq \infty .$ and X is the  ball quasi-Banach function space.
	The Hardy local Morrey-type spaces associated with ball quasi-Banach function spaces $H  L M_{X, q}^{\lambda}\left(\mathbb{R}^{n}\right)$ collects all $f \in \mathcal{S}^{\prime}\left(\mathbb{R}^{n}\right)$ such that $\sup _{t>0}\left|e^{t \Delta} f\right| \in  L M_{X, q}^{\lambda}\left(\mathbb{R}^{n}\right)$.
	$$
	\|f\|_{H  L M_{X, q}^{\lambda}} :=\left\|\sup _{t>0}\left|e^{t \Delta} f\right|\right\|_{ L M_{X, q}^{\lambda}} < \infty.
	$$
\end{definition}
\begin{proposition}
	For  $\lambda \geq 0  $ , $ 0 <q,q_0,q_1 \leq \infty ,$ $f \in L^{loc}_{1}$ and X is the  ball quasi-Banach function space. If $q_0<q_1$, then $L M_{X, q_0}^{\lambda}\subset L M_{X, q_1}^{\lambda}$ and $\widetilde{L M}_{X, q_0}^{\lambda}\subset \widetilde{L M}_{X, q_1}^{\lambda}$.
\end{proposition}
\begin{proof}
1. First, suppose that  $q_{1}=\infty $, then
	$$
	\begin{aligned}
	\|f\|_{L M_{X, \infty}^{\lambda}} &=\sup _{r>0} r^{-\lambda} \|f \chi_{B(0,r)}\|_{X}\\
	&=\left(\lambda q_{0}\right)^{\frac{1}{q_{0}}} \sup _{r>0}\left(\int_{r}^{\infty} t^{-\lambda q_{0}} \frac{d t}{t}\right)^{\frac{1}{q_{0}}} \|f \chi_{B(0,r)}\|_{X} \\
	&\leq\left(\lambda q_{0}\right)^{\frac{1}{q_{0}}} \sup _{r>0}\left(\int_{r}^{\infty}\left(t^{-\lambda} \|f \chi_{B(0,t)}\|_{X}\right)^{q_{0}} \frac{d t}{t}\right)^{\frac{1}{q_{0}}}\\&=\left(\lambda q_{0}\right)^{\frac{1}{q_{0}}}\|f\|_{L M^{\lambda}_{X, q_{0}}} 
	\end{aligned}
	$$
	If  $q_{1}<\infty $, then it suffices to apply the interpolation inequality
	$$
	\begin{aligned}
	\|f\|_{L M_{X, q_{1}}^{\lambda}} &\lesssim\|f\|_{L M_{X, \infty}^{\lambda}}^{1-\frac{q_{0}}{q_{1}}}\|f\|_{L M_{X, q_{0}}^{\lambda}}^{\frac{q_{0}}{q_{1}}}\\
	& \leq \left(\left(\lambda q_{0}\right)^{\frac{1}{q_{0}}}\|f\|_{L M_{X, q_{0}}} \right)^{1-\frac{q_{0}}{q_{1}}} \|f\|_{L M_{X, q_{0}}^{\lambda}}^{\frac{q_{0}}{q_{1}}}\\
	&\leq \left(\lambda q_{0}\right)^{\frac{1}{q_{0}}-\frac{1}{q_{1}}}\|f\|_{L M_{X, q_{0}}^{\lambda}}
	\end{aligned}
	$$
	2. First, suppose that  $q_{1}=\infty $, then	
	$$
	\begin{aligned}
	\|f\|_{{\widetilde{L M}}_{X, \infty}^\lambda}
	&=\sup_{r>0} r^{-\lambda} r^n\sup _{\rho \geq r} \rho^{-n}  \|f \chi_{B(0,t)}\|_{X}\quad ~~~\quad~~~\quad~~~  \quad~~~ \quad~~~
			\end{aligned}
	$$
	$$
	\begin{aligned}
	\quad~~~ \quad~~~ \quad~~~  \quad \quad ~~~
	&=\left( (n-\lambda) q_0 \right)^{\frac{1}{q_0}}
		V_n \sup_{r>0} 
		\left( \int_0^r s^{(n-\lambda)q_0}\frac{ds}{s}\right)^{\frac{1}{q_0}} \sup _{\rho \geq r}\rho^{-n}  \|f \chi_{B(0,\rho)}\|_{X}
	\\&\leq \left( (n-\lambda) q_0 \right)^{\frac{1}{q_0}}
	V_n \sup_{r>0} \left( \int_0^r \left( s^{(n-\lambda)} \sup _{\rho \geq r}\rho^{-n}  \|f \chi_{B(0,t)}\|_{X}\right)^{q_0}\frac{ds}{s}\right)^{\frac{1}{q_0}}
	\\&=\left( (n-\lambda) q_0 \right)^{\frac{1}{q_0}} 	\|f\|_{{\widetilde{L M}}_{X, q_0}^\lambda}
	\end{aligned}
	$$
	If $  q_{1}<\infty $, it is proved in a similar way as above.	
\end{proof}	

\begin{remark} \cite{W2021}
	$L_p(\R^n)$, $L_{\vec{p}}\left(\mathbb{R}^n\right)$, $L_{p(\cdot)}\left(\mathbb{R}^n\right)$ are proven to be  ball quasi-Banach function spaces
\end{remark}

\begin{remark}

In the rest of the article, since the properties and applications of $L M_{X, q_{0}}^{\lambda}(\R^n)$ will be discussed when X  is $L_p(\R^n)$, $L_{\vec{p}}(\R^n)$, $L_{p(\cdot)}(\R^n)$, all satisfied $\sigma:= log_r  \|\chi_{B(0,r)}\|_{X} $ and X is the  ball quasi-Banach function space. 

\end{remark}
\begin{remark}
When X is $L_p$, for  $L M_{X, q}^{\lambda}$ back to  local Morrey-type spaces \cite{BN2010} $L M_{p, q}^{\lambda}$. In particular, because of Remark 2.5, we also obtain the norm equivalence of $L M_{p, q}^{\lambda}$ and homogeneous  Herz spaces  \cite{LS2008} $\dot{K}_{q}^{\alpha, p}(\mathbb{R}^n)$.

When X is $L_{\vec{p}}$, for  $L M_{X, q}^{\lambda}$ back to  local mixed Morrey Spaces \cite{ZH2021} $L M_{\vec{p}, q}^{\lambda}$. In particular, because of Remark 2.5, we also obtain the norm equivalence of $L M_{\vec{p}, q}^{\lambda}$ and homogeneous  Mixed Herz spaces  \cite{MW2021} $\dot{K}_{\vec{q}}^{\alpha, p}(\mathbb{R}^n)$.

When X is $L_{p(\cdot)}$, for  $L M_{X, q}^{\lambda}$ obtains local Morrey-type spaces with variable exponents $L M_{{p(\cdot)}, q}^{\lambda}$.
In particular, because of Remark 2.5, we also obtain the norm equivalence of $L M_{{p(\cdot)}, q}^{\lambda}$ and homogeneous  variable exponents Herz spaces  \cite{AA2012} $\dot{K}_{q(\cdot)}^{\alpha, p}(\mathbb{R}^n) ~~(\alpha(\cdot)=\alpha)$.
\end{remark}
\begin{proposition}
	Let $  0<q \leq \infty $,  $0 \leq \lambda <\sigma   $. Then  $L M_{X, q}^{\lambda}\left(\mathbb{R}^{n}\right) \hookrightarrow   \mathcal{S}^{\prime}\left(\mathbb{R}^{n}\right)$  in the sense of continuous embedding.
\end{proposition}
\begin{proof}
	Denote by $\mathcal{B}_{x}$ the set of all open balls in $\mathbb{R}^{n}$ which contain $x$. 
	The Hardy-Littlewood maximal operator $M$ is bounded on $L M_{X, q}^{\lambda}\left(\mathbb{R}^{n}\right)$ (As demonstrated in Section 3 of this paper). Therefore,
	$$
	\frac{\alpha}{|B(R)|} \int_{B(R)}|f(y)| d y \leq\left\|\chi_{B(1)} M f\right\|_{L M_{X, q}^{\lambda}\left(\mathbb{R}^{n}\right)} \leq\|M f\|_{L M_{X, q}^{\lambda}\left(\mathbb{R}^{n}\right)} \lesssim \|f\|_{L M_{X, q}^{\lambda}\left(\mathbb{R}^{n}\right)},
	$$
	where $\alpha \equiv\left\|\chi_{B(1)}\right\|_{L M_{X, q}^{\lambda}\left(\mathbb{R}^{n}\right)}$. Then for all $\kappa \in \mathcal{S} (\R^n)$ and $f \in L M_{X, q}^{\lambda}\left(\mathbb{R}^{n}\right) $
	$$
	\begin{aligned}
	\int_{\mathbb{R}^{n}}|\kappa(x) f(x)| d x &=\int_{B(1)}|\kappa(x) f(x)| d x+\sum_{j=1}^{\infty} \int_{B(j+1) \backslash B(j)}|\kappa(x) f(x)| d x \\
	& \leq\|\kappa\|_{L^{\infty}(B(1))}\|f\|_{L^{1}(B(1))}+\sum_{j=1}^{\infty} \int_{B(j+1) \backslash B(j)} \frac{|x|^{2 n+1}}{j^{2 n+1}}|\kappa(x) f(x)| d x \\
	& \lesssim \|f\|_{L M_{X, q}^{\lambda}\left(\mathbb{R}^{n}\right)}\left(\sup _{x \in \mathbb{R}^{n}}(1+|x|)^{2 n+1}|\kappa(x)|\right)
	\end{aligned}
	$$
\end{proof}

\begin{Assumption}
If the Hardy-Littlewood maximal operator  $M$ 
is bounded from $X \left(\mathbb{R}^{n}\right)$ to $X \left(\mathbb{R}^{n}\right)$, then $X \in \mathbb{M}$.
\end{Assumption}
\begin{proposition}
The local Morrey space associated with ball quasi-Banach function spaces $L M_{X, q}$ is called a ball quasi-Banach function space.
\end{proposition}

The proof is simple, interested readers can prove it among themselves.\par

 \begin{proposition}
 	Let $  1 \leq q <\infty $ and $0\leq \lambda<\sigma$ and X is the  ball quasi-Banach function space. Then for any $ w \subset \mathbb{R}^{n}$ 
 	$$
 	\|f\|_{L M^\lambda _{X q }} \sim
 	\left(\sum_{j=-\infty}^{\infty}\left(2^{-\lambda j} \|f\chi_{ B(0,  2^{j})}\|_{X } \right)^{q}\right)^{1 / q} 
 	$$
 \end{proposition}
 \begin{proof}
 	We start with the equality
 	$$	
 	\begin{aligned}
 	\|f\|_{L M^\lambda _{X q }} &=\left(\int_{0}^{\infty}\left(t^{-\lambda} \|f\chi_{ B(0,  2^{j})}\|_{X }\right)^{q} \frac{\mathrm{d} t}{t}\right)^{1 / q} \\
 	&=\left(\sum_{j=-\infty}^{\infty} \int_{2^{j}}^{2^{j+1}}\left(t^{-\lambda} \|f\chi_{ B(0,  2^{j})}\|_{X }\right)^{q} \frac{\mathrm{d} t}{t}\right)^{1 / q}.
 	\end{aligned}
 	$$
 	On one hand
 	$$
 	\begin{array}{l}
 	\left(\sum_{j=-\infty}^{\infty} \int_{2^{j}}^{2^{j+1}}\left(t^{-\lambda} \|f\chi_{ B(0,  2^{j})}\|_{X }\right)^{q} \frac{\mathrm{d} t}{t}\right)^{1 / q} \\
 	\leq 2^{\lambda}(\ln 2)^{1 / q}\left(\sum_{j=-\infty}^{\infty}\left(2^{-\lambda(j+1)}\|f\chi_{ B(0,  2^{j})}\|_{X }\right)^{q}\right)^{1 / q}
 	\end{array}
 	$$
 	On the other hand
 	$$
 	\begin{array}{l}
 	\left(\sum_{j=-\infty}^{\infty} \int_{2^{j}}^{2^{j+1}}\left(t^{-\lambda}\|f\chi_{ B(0,  2^{j})}\|_{X } \right)^{q} \frac{\mathrm{d} t}{t}\right)^{1 / q} \\
 	\geq 2^{-\lambda}(\ln 2)^{1 / q}\left(\sum_{j=-\infty}^{\infty}\left(2^{-\lambda j}\|f\chi_{ B(0,  2^{j})}\|_{X }\right)^{q}\right)^{1 / q}
 	\end{array}
 	$$
 	and we obtain the required equivalence.
 	
 \end{proof}
Let $B_j=\{ x\in \mathbb{R}^n : |x| \leq 2^j\}$ and $A_j=B_j \setminus B_{j-1} $ for any $k \in \mathbb{Z}$. Denote $ \chi_j=\chi_{A_j}$, where $\chi_E$ is the charcteristic function of set E.
\begin{remark}
	Let $ 1<q\leq \infty$, $0\leq \lambda < \sigma$ and 	for all measurable functions $f: \mathbb{R}^{n} \rightarrow \mathbb{C}$.
	Then
	$$
	\|f\|_{L M_{X, q}^{\lambda}} \sim \left(\sum_{j=-\infty}^{\infty}2^{-j \lambda }\|f\chi_j\|_{X}^q\right)^{\frac{1}{q}}.
	$$
	
\end{remark}

\begin{proof}
	It is clear from that
	$$
	\|f\|_{L M_{\vec{p} \theta}^{\lambda}} \gtrsim\left\{\sum_{j=-\infty}^{\infty}\left(2^{-\lambda j}\|f\chi_j\|_{X}\right)^{q}\right\}^{\frac{1}{q}}.
	$$
	To prove the reverse estimate, 
	$$\begin{aligned}
	\|f\|_{L M_{\vec{p} \theta}^{\lambda}}& \sim 	\left(\sum_{j=-\infty}^{\infty}\left(2^{-\lambda j} \|f\chi_{ B(0,  2^{j})}\|_{X } \right)^{q}\right)^{1 / q} \\
	&=\left\{\sum_{j=-\infty}^{\infty}\left(\sum_{k=-\infty}^{j} 2^{-\lambda j}\|f\chi_j\|_{X}\right)^{\theta}\right\}^{\frac{1}{\theta}}\\
	&=\left\{\sum_{j=-\infty}^{\infty}\left(\sum_{k=-\infty}^{\infty} \chi_{(-\infty, j]}(k) 2^{-\lambda j}\|f\chi_j\|_{X}\right)^{q}\right\}^{\frac{1}{q}}\\
	&\leq \sum_{k=-\infty}^{\infty}\left\{\sum_{j=-\infty}^{\infty}\left(\chi_{(-\infty, j]}(k) 2^{-\lambda j}\|f\chi_j\|_{X}\right)^{q}\right\}^{\frac{1}{q}}\\
	&=\left\{\sum_{k=-\infty}^{\infty}\left(\frac{1}{1-2^{-\lambda}} \cdot 2^{-\lambda k}\|f\chi_j\|_{X}\right)^{q}\right\}^{\frac{1}{q}}.
	\end{aligned}
	$$
	
\end{proof}
\section{Boundedness of Hardy-littlewood Operators on the local Morrey space associated with ball quasi-Banach function spaces}
The Hardy operator $H$  and its dual operators $ H^{*} $, given by:
$$
H g(r)=\int_{0}^{r} g(t) d t \text ,~~\quad ~ H^{*} g(r)=\int_{r}^{\infty} g(t) d t .
$$
Since the following article requires, we must need the following relationship.
For  $1 \leq q <\infty $ and a measurable function $ v:(0, \infty) \rightarrow(0, \infty) $, whose norm is given by
$$
\|f\|_{L_{q, v}(0, \infty)} := \| v f\|_{L_{q}(0, \infty)}.
$$

	Consider first the following "partial" maximum function.
$$
\underline{M}f = \mathop{sup}\limits_{0<t \leq r} \frac{1}{|B(x,r)|} \int_{B(x,r)} |f(y)| dy;$$
$$
\overline{M}f = \mathop{sup}\limits_{t > r} \frac{1}{|B(x,r)|} \int_{B(x,r)} |f(y)| dy.\\
$$
\begin{lemma}
	Let  $f\in L_{1}^{loc}(\mathbb{R}^{n})$, then for  $B(y,2r)$ in $\mathbb{R}^n$
	$$\| 	M\left(f \chi_{\complement_{B(x, 2 r)}} \right) \chi_{B(0,t)}\|_{X} \gtrsim  r^{\sigma}  \overline {M} f(x).
	$$
\end{lemma}
\begin{proof}
	If $y \in B(x,r), B(x,\frac{t}{2}) \subset B(y,t) \cap \complement B(x,2r) .$
	$$
	\begin{aligned}
	M\left(f \chi_{\complement_{B(x, 2 r)}}\right)(y)&=\sup _{t>0} \frac{1}{|B(y, t)|} \int_{B(y, t) \cap{ }^{\complement  B(x,2r)}}|f(z)| \mathrm{d} z\\
	&\gtrsim \sup _{t \geq 2r} \frac{1}{|B(x, 2 t)|} \int_{B(x, 2 t)}|f(y)| \mathrm{d} y=\overline {M} f(x).
	\end{aligned}
	$$
	$$
	\| 	M\left(f \chi_{\complement_{B(x, 2 r)}} \right) \chi_{B(0,t)}\|_{X} \gtrsim  r^{\sigma}  \overline {M} f(x).
	$$
\end{proof}
\begin{lemma}
	Let   $f\in L_{1}^{loc}(\mathbb{R}^{n})$, then
	$$
	\| Mf \chi_{B(0,t)}\| _{X} = \|M (f \chi_{B(x,2r)}) \chi_{B(0,t)}\| _{X} + r^{\sigma} \overline{M} f(x).
	$$
\end{lemma}
\begin{proof}
	It is obvious that for  $B(x,r)$,
	$$
	\| Mf \chi_{B(0,t)}\| _{X} \leq \|M (f \chi_{B(x,2r)}) \chi_{B(0,t)}\| _{X} +  \|M (f \chi_{\complement B(x,2r)}) \chi_{B(0,t)} \| _{X}.
	$$
	\par
	Let $y \in B(x,r)$. If $B(y,t) \cap \complement B (x,2r) \ne \emptyset $, then $t>r$. Indeed $z \in B(y,t) \cap \complement B(x,2r) \ne \emptyset$, then $t > |z-y| \ge |z-x|-|x-y| >2r -r <r $.\par Anothers, $B(y,t) \cap \complement B(x,2r) \subset  B(x,2t)$. Indeed , if $z \in B(y,t) \cap \complement B(x,2r)$ , then $|z-x| \leq |z-y|+|y-x| <t+r<2t.$\par
	Hence
	$$
	\begin{aligned}
	M\left(f \chi_{\complement_{B(x, 2 r)}}\right)(y)&=\sup _{t>0} \frac{1}{|B(y, t)|} \int_{B(y, t) \cap{ }^{\complement  B(x,2r)}}|f(z)| \mathrm{d} z\\
	&	\lesssim \sup _{t \geq r} \frac{1}{|B(x, 2 t)|} \int_{B(x, 2 t)}|f(y)| \mathrm{d} y=\overline {M} f(x).
	\end{aligned}
	$$
	
	$$
	\| 	M\left(f \chi_{\complement_{B(x, 2 r)}}\right) \chi_{B(0,t)}\|_{X} \lesssim  r^{\sigma}  \overline {M} f(x).
	$$\par
	On the one hand,
	$$
	\| M\left(f \chi_{{B(x, 2 r)}}\right) \chi_{B(0,t)}\|_{X} \leq \| Mf \chi_{B(0,t)}\|_{X}.
	$$
	\par 
	On the other hand, if $y\in B(x,r)  , ~z \in  B(y,t) \cap \complement B(x,2r)  $ and Lemma 3.1 then
	$$
	\| Mf \chi_{B(0,t)}\|_{X} \gtrsim \| M(f \chi_{\complement_B(x,2r)}) \chi_{B(0,t)}\| _{X}\gtrsim   r^{\sigma} \overline {M} f(x).
	$$
	The proof is complete.
\end{proof}
\begin{lemma}
	Let  $f\in L_{1}^{loc}(\mathbb{R}^{n})$, $X \in \mathbb{M}$. Then for any  $B(x,r)$ in $\mathbb{R}^n,$
	$$\|Mf \chi_{B(0,t)}\|_{X}\lesssim r^{\sigma}\int_{r}^{\infty}\|f \chi_{B(0,t)}\|_{X}\frac{dt}{t^{\sigma+1}}.$$
	
\end{lemma}
\begin{proof}
	By Lemma 3.2
	$$\begin{aligned}
	\| Mf \chi_{B(0,t)}\| _{X} \leq \|M (f \chi_{B(x,2r)}) \chi_{B(0,t)}\| _{X} + r^{\sigma} \overline{M} f(x)
	&:=I+II
	\end{aligned}
	$$
	On the one hand,
	$$\begin{aligned}
	I & \leq \|M (f \chi_{B(x,r)}) \| _{X } \leq \|f \chi_{B(y,2r)} \| _{X}
	\leq \| f \chi_{B(y,2r)}\| _{X }  \\
	& \lesssim  r^{\sigma}\int_{2r}^{\infty}\|f \chi_{B(x,2r)}\|_{X }\frac{dt}{t^{\sigma+1}}
	\lesssim  r^{\sigma}\int_{r}^{\infty}\|f \chi_{B(x,t)}\|_{X }\frac{dt}{t^{\sigma+1}}.
	\end{aligned}
	$$
	On the other hand,
	$$
	\begin{aligned}
	II & =r^{\sigma}  \mathop{sup}\limits_{t > r} \frac{1}{|B(x,r)|} \int_{B(x,r)} |f(y)| dy \\
	& \lesssim  r^{\sigma}\int_{r}^{\infty}\|f \chi_{B(x,t)}\|_{X }\frac{dt}{t^{\sigma+1}}.
	\end{aligned}
	$$
	
\end{proof}
	\begin{lemma}
	Let 
	$ 0<q \leq \infty $, $0 \leq \lambda < \sigma$, $X \in \mathbb{M}$,	for all $f \in L^{loc}_{1}$, where
	$$
	g_{X}(t)= \| f \chi_{B(0,t^{-\frac{1}{\sigma}})}\| _{X },
	$$
	$$
	v(r) = r^{\frac{\lambda}{\sigma}-1-\frac{1}{q}}.$$ Then
	$$
	\| Mf \| _{L M_{X, q}^{\lambda}} \lesssim \|Hg_{X}\| _{L_{q,v}(0, \infty)},
	$$

\end{lemma}
\begin{proof}
	By Lemma 3.3,
	$$
	\begin{aligned}
	\|Mf\|_{L M_{X, q}^\lambda}
	&=\| r^{-\lambda-\frac{1}{q} }\|Mf \chi_{B(0,r)}\|_{X}\|_{L_{q}(0,\infty)}\\
	&\lesssim\|  r^{-\lambda}r^{\sigma}\int_{r}^{\infty}\|f \chi_{B(0,t)}\|_{X} \frac{dt}{t^{\sigma+1}}\|_{L_{q}(0,\infty)}\\
	&\lesssim\| r^{-\lambda}r^{\sigma}\int_{0}^{r^{-\frac{1}{\sigma}}}\|f \chi_{B(0,t^{-\frac{1}{\sigma}})}\|_{X }dt \|_{L_{q}(0,\infty)}\\
	&\lesssim\| r^{\frac{\lambda}{\sigma}-1-\frac{1}{q}} Hg_{X}(r)\|_{L_{q}(0,\infty)}\\
	&\sim \|Hg_{X}\|_{L{q,\nu}(0,\infty)}.
	\end{aligned}
	$$
\end{proof}
	\begin{lemma}
	Let 
	$ 0<q \leq \infty $, for all $f \in L^{loc}_{1}$,  
	where
$$
	g_{X}(t)= \| f \chi_{B(0,t^{-\frac{1}{\sigma}})}\| _{X },
$$
	$$
	v(r) = r^{\frac{\lambda}{\sigma}-\frac{1}{q}}.$$
	Then
	$$
	\|g_{X}\| _{L_{q,v}(0, \infty)} \lesssim  \| f \| _{L M_{X, q}^{\lambda}},
	$$
\end{lemma}
	$$
\begin{aligned}
\|g_{X}\|_{L{q,v}(0,,\infty)}&= \|v(r)\| f \chi_{B(0,r^{-\frac{1}{\sigma}})}\|_{X }\|_{L_{q}(0,\infty)}\\
			\end{aligned}
$$
$$
\begin{aligned}
\quad~~~ \quad~~~\quad~~~ \quad
&=\|r^{\frac{\lambda}{\sigma}-\frac{1}{q}}\|f \chi_{B(0,r^{-\frac{1}{\sigma}})}\|_{X} \|_{L_{q}(0,\infty)}\\
&\lesssim \| r^{-\lambda-\frac{1}{q}}\|f \chi_{B(0,r)}\|_{X }\|_{L_{q}(0,\infty)}\\
&=\|f\| _{L M_{X, q}^{\lambda}}.
\end{aligned}
$$
\begin{lemma}
Let 
$ 0<q \leq \infty $, $0 \leq \lambda < \sigma$, for all $f \in L^{loc}_{1}$, 
$$
g_{X}(t)= \| f \chi_{B(0,t^{-\frac{1}{\sigma}})}\| _{X },
$$
$$
v_2(r) = r^{\frac{\lambda}{\sigma}-1-\frac{1}{q}}$$
$$
v_1(r) = r^{\frac{\lambda}{\sigma}-\frac{1}{q}}.$$\\
Then
$$
\|Hg_{X}\| _{L_{q,v_2}(0, \infty)} \lesssim \|g_{X}\|_{L{q,v_1}(0,,\infty)},
$$
\end{lemma}
\begin{proof}
$$
\begin{aligned}
\|Hg_{X}\| _{L_{q,v_2}(0, \infty)}&
=\left( \int_{0}^{\infty} \left( r^{-1-\frac{1}{q}+\frac{\lambda}{\sigma}} \int_{0}^r \|f \chi_{B(0,t^{-\frac{1}{\sigma}})}\|_{X } dt \right)^q dr\right)^{\frac{1}{q}}
\\&\leq \left( \int_{0}^{\infty} \left( \|f \chi_{B(0,t^{-\frac{1}{\sigma}})}\|_{X } \int_{0}^t r^{-1-\frac{1}{q}+\frac{\lambda}{\sigma}}  dr \right)^q dt\right)^{\frac{1}{q}}
\\& \leq \left(-\frac{1}{q}+\frac{\lambda}{\sigma}\right)^{-1}
\left( \int_{0}^{\infty} \left(\|f \chi_{B(0,t^{-\frac{1}{\sigma}})}\|_{X }   t^{-\frac{1}{q}+\frac{\lambda}{\sigma}}   \right)^q dt\right)^{\frac{1}{q}}
\\&\lesssim \|g_{X}\|_{L{q,v_1}(0,,\infty)}.
\end{aligned}
$$
\end{proof}
	\begin{theorem}
	Let  $0 \leq \lambda< \sigma, 0<q\leq \infty$, $X \in \mathbb{M}$. 
	Then the operator $ M $ is bounded from $ L M_{X, q}^\lambda $ to $ L M_{X, q}^\lambda $.
\end{theorem}
\begin{proof}
By Lemma 3.4, 3.5 and 3.6.
$$
	\| Mf \| _{L M_{X, q}^{\lambda}} \lesssim \|Hg_{X}\| _{L_{q,v}(0, \infty)} \lesssim \|g_{X}\|_{L{q,v_1}(0,,\infty)}\lesssim  \| f \| _{L M_{X, q}^{\lambda}}
$$
\end{proof}
\section{Interpolation theorem}
\begin{lemma}
	Let $ 0< q \leq \infty, 0 \leq \lambda < \sigma$. Then 
	
	$$
	\widetilde{L M}_{X, q}^{\lambda}=L M_{X, q}^{\lambda}$$
	Moreover, for  $ q<\infty $, 
	$$
	\left(\frac{nq-\lambda q}{n}\right)^{\frac{1}{q}}\|f\|_{\widetilde{L M}_{X, q}^{\lambda}} \leq\|f\|_{L M_{X, q}^{\lambda}} \leq\|f\|_{\widetilde{L M}_{X, q}^{\lambda}}.
	$$
	For $ q=\infty $,
	$$
	\|f\|_{\widetilde{L M}_{X, \infty}^{\lambda}} =\|f\|_{L M_{X,\infty}^{\lambda}}. 
	$$ 
\end{lemma}

\begin{proof}
	For $ q<\infty $,
	$$
	\begin{aligned}
	\|f\|_{\widetilde{L M}_{X, q}^{\lambda}}^{q}&=\int_{0}^{\infty}\left(r^{-\lambda+n}\sup _{\rho \geq r}\rho^{-n} \|f \chi_{ B(0,\rho)}\|_{X}\right)^{q} \frac{d r}{r} \\
	&=n \int_{0}^{\infty}\left(r^{-\lambda+n} \sup _{\rho \geq r}  \int_{\rho}^{\infty} t^{-n} \frac{d t}{t}  \|f \chi_{ B(0,\rho)}\|_{X} \right)^{q} \frac{d r}{r} \\&
	\leq n \int_{0}^{\infty} r^{-\lambda q+q n} \sup _{\rho \geq r}\left(\int_{\rho}^{\infty} t^{-qn}
	\|f \chi_{ B(0, t)}\|_{X }^q \frac{d t}{t}\right) \frac{d r}{r} \\
	&=n\int_{0}^{\infty} r^{\left(-\lambda+n\right) q}\left(\int_{r}^{\infty} t^{-qn}	\|f \chi_{ B(0, t)}\|_{X }^q \frac{d t}{t}\right) \frac{d r}{r} \\
	&=n \int_{0}^{\infty} t^{-qn}	\|f \chi_{ B(0, t)}\|_{X }^q\left(\int_{0}^{t} r^{\left(-\lambda+n\right) q} \frac{d r}{r}\right) \frac{d t}{t} \\
	&=\frac{n}{nq-\lambda q}\|f\|_{L M_{X, q}^{\lambda}}^{q}.
	\end{aligned}$$
	If  $q=\infty $, then	
	$$\begin{aligned}
	\|f\|_{\widetilde{L M}_{X, \infty}^{\lambda}}&=\sup_{r>0} r^{-\lambda} \left\| f \chi_{B(0, t)}\right\|_{\widetilde{X}} \\
	&=\sup _{r>0} r^{-\lambda} r^{n} \sup_{r\leq t} t^{-n}\|f \chi_{ B(0, t)}\|_{X} \\
	&=\sup _{t>0} t^{-\lambda}\|f \chi_{ B(0, t)}\|_{X}\\
	&=\|f\|_{L M_{X, \infty}}^\lambda.
	\end{aligned}$$
\end{proof}
\begin{theorem}
	Let $ 0<q_{0}, q_{1}, q \leq \infty $, $0 \leq \lambda_{0},\lambda_{1},\lambda < \sigma$,  $0<\theta<1 $ and  $ \lambda=(1-\theta) \lambda_{0}+\theta \lambda_{1} $. Then
	
	$$\left(L M_{X, q_{0}}^{\lambda_{0}}, L M_{X, q_{1}}^{\lambda_{1}}\right)_{\theta, q}=L M_{X, q}^{\lambda}$$

\end{theorem}
\begin{proof}
	1. If $q<\infty$, 1.1 first, let us prove that
	$$
	\left(L M_{X, q_{0}}^{\lambda_{0}}, L M_{X, q_{1}}^{\lambda_{1}}\right)_{\theta, q} \subset L M_{X, q}^{\lambda}	
	$$
	for $ \lambda_{0}<\lambda_{1} $. Let $ f \in\left(L M_{X, q_{0}}^{\lambda_{0}}, L M_{X, q_{1}}^{\lambda_{1}}\right)_{\theta, q} $ and $  f=\varphi+\psi $ with  $\varphi \in L M_{X, q_{0}}^{\lambda_{0}}  $ and $ \psi \in L M_{X, q_{1}}^{\lambda_{1}} $. According to Proposition 2.1 and Lemmas 4.1, then 
	$$
	\begin{aligned}
	r^{-\lambda}\|f \chi_{ B(0, r)}\| _{\widetilde{X}}&\lesssim r^{-\lambda}\left(\|\varphi \chi_{ B(0, r)}\|_{\widetilde{X}}+\|\psi \chi_{ B(0, r)}\|_{\widetilde{X}}\right)\\
	& = r^{\lambda_{0}-\lambda}\left( r^{-\lambda_0}\|\varphi \chi_{ B(0, r)}\|_{\widetilde{X}}+r^{\lambda_{1}-\lambda_{0}} r^{-\lambda_{1}}\|\psi \chi_{ B(0, r)}\|_{\widetilde{X}}\right) \\
	&\lesssim  r^{\lambda_{0}-\lambda}\left( \mathop{sup}\limits_{s>0} s^{-\lambda_0}\|\varphi \chi_{ B(0, r)}\|_{\widetilde{X}}+r^{\lambda_{1}-\lambda_{0}}\mathop{sup}\limits_{s>0} s^{-\lambda_{1}}\|\psi \chi_{ B(0, r)}\|_{\widetilde{X}}\right)\\
	&=r^{\lambda_{0}-\lambda}\left(\|\varphi\|_{\widetilde{L M}_{X, \infty}^{\lambda_{0}}}+r^{\lambda_{1}-\lambda_{0}}\|\psi\|_{\widetilde{L M}_{X, \infty}^{\lambda}}\right)\\
				\end{aligned}
	$$
	$$
	\begin{aligned}
	\quad~~~ \quad~~~\quad~
	&\lesssim  r^{\lambda_{0}-\lambda}\left(\|\varphi\|_{L M_{X, \infty}^{\lambda_{0}}}+r^{\lambda_{1}-\lambda_{0}}\|\psi\|_{L M_{X, \infty}^{\lambda}}\right)\\
&\lesssim r^{\lambda_{0}-\lambda}\left((\lambda_{0} q_0)^{\frac{1}{q_0}}\|\varphi\|_{L M_{X, q_0}^{\lambda_{0}}}+(\lambda_{1} q_1)^{\frac{1}{q_1}}r^{\lambda_{1}-\lambda_{0}}\|\psi\|_{L M_{X, q_1}^{\lambda_1}}\right)\\
	&\lesssim r^{-\theta (\lambda_{1}-\lambda_{0})}\left(\|\varphi\|_{L M_{X,q_0}^{\lambda_{0}}}+r^{\lambda_{1}-\lambda_{0}}\|\psi\|_{L M_{X, q_1}^{\lambda_1}}\right)
	\end{aligned}
	$$
	Since the representation $ f=\varphi+\psi$  is arbitrary, 
	$$
	r^{-\lambda}\|f\|_{\widetilde{X}(B(0, r))} \lesssim  r^{-\theta \left(\lambda_{1}-\lambda_{0}\right)} K\left(r^{\lambda_{1}-\lambda_{0}}, f\right),
	$$
	where \cite{JE1972}
	$$
	K(t,f)=\inf _{\substack{f=\varphi+\psi \\ \varphi \in L M_{X, q_{0}}^{\lambda_{0}}, \psi \in L M_{X, q_{1}}^{\lambda_{1}}}}\left(\|\varphi\|_{L M_{X, q_{0}}^{\lambda_{0}}}+t\|\psi\|_{L M_{X, q_{1}}^{\lambda_{1}}}\right), \quad t>0 .
	$$
	Hence,	
	$$
	\begin{aligned}
	\|f\|_{L M_{X, q}^{\lambda}} &\leq\|f\|_{\widetilde{L M}_{X, q}^{\lambda}}=\left(\int_{0}^{\infty}\left(r^{-\lambda}\|f\chi_{ B(0, r)}\|_{\widetilde{X}}\right)^{q} \frac{d r}{r}\right)^{\frac{1}{q}} \lesssim \left(\int_{0}^{\infty}\left(r^{-\theta\left(\lambda_{1}-\lambda_{0}\right)} K\left(r^{\lambda_{1}-\lambda_{0}}, f\right)\right)^{q} \frac{d r}{r}\right)^{\frac{1}{q}}\\ 
&\sim \left(\lambda_{1}-\lambda_{0}\right)^{-\frac{1}{q}}\left(\int_{0}^{\infty}\left(t^{-\theta} K(t, f)\right)^{q} \frac{d t}{t}\right)^{\frac{1}{q}}\sim \left(\lambda_{1}-\lambda_{0}\right)^{-\frac{1}{q}}\|f\|_{\left(L M_{X, q_{0}}^{\lambda_{0}}, L M_{X, q_{1}}^{\lambda_{1}}\right)_{\theta, q}}.
\end{aligned} $$
	
	1.2. On the contrary, let us prove that
	$$
	L M_{X, q}^{\lambda} \subset\left(L M_{X, q_{0}}^{\lambda_{0}}, L M_{X, q_{1}}^{\lambda_{1}}\right)_{\theta, q}
	$$
	for  $\lambda_{0}<\lambda_{1} .$ Set  $\widetilde{q}_{0}=\min \left\{q_{0}, q\right\} $ and $ \widetilde{q}_{1}=\min \left\{q_{1}, q\right\}$. It suffices to prove that
	$$
	L M_{X, q}^{\lambda} \subset\left(L M_{X, \widetilde{q}_{0}}^{\lambda_{0}}, L M_{X, \widetilde{q}_{1}}^{\lambda_{1}}\right)_{\theta, q},
	$$
	because
	$$
	\left(L M_{X, \widetilde{q_{0}}}^{\lambda_{0}}, L M_{X, \widetilde{q_{1}}}^{\lambda_{1}}\right)_{\theta, q} \subset\left(L M_{X, q_{0}}^{\lambda_{0}}, L M_{X, q_{1}}^{\lambda_{1}}\right)_{\theta, q}
	$$
	by Proposition 2.1. Therefore, without loss of generality we assume that $ 0<q_{0}, q_{1} \leq q \leq \infty $.
	\par 
	First, suppose that $ q<\infty$.
	\par 
	1.2.1. Let $ f \in L M_{X, q}^{\lambda} $ and $ t>0 $. For $ x \in \mathbb{R}^{n}$ , set
	
	\begin{equation} \label{222}
	\varphi_{t}(x)=\left\{\begin{array}{ll}
	f(x) & \text { if } x \in B(0, t), \\
	0 & \text { if } x \notin B(0, t)
	\end{array} \quad \text { and } \quad \psi_{t}(x)=f-\varphi_{t}(x) .\right.
	\end{equation}
	Then, applying the change of variables $  t^{\lambda_{1}-\lambda_{0}}=s $, 
	$$
	\begin{aligned}
	\|f\|_{\left(L M_{X, q_{0}}^{\lambda_{0}}, {L M}_{X, q_{1}}^{\lambda_{1}}\right)_{\theta, q}} &\leq\|f\|_{\left(\widetilde{L M}_{X, q_0}^{\lambda{0}}, \widetilde{L M}_{X, q_{1}}^{\lambda_{1}}\right)_{\theta, q} }=\left(\int_{0}^{\infty}\left(s^{-\theta} \widetilde{K}(s, f)\right)^{q} \frac{d s}{s}\right)^{\frac{1}{q}}\\
	&=\left(\lambda_{1}-\lambda_{0}\right)^{\frac{1}{q}}\left(\int_{0}^{\infty}\left(t^{-\left(\lambda_{1}-\lambda_{0}\right) \theta} \tilde{K}\left(t^{\lambda_{1}-\lambda_{0}}, f\right)\right)^{q} \frac{d t}{t}\right)^{\frac{1}{q}}.
	\end{aligned}
	$$
	where
	$$
	\begin{aligned}
	\widetilde{K}(t^{\lambda_{1}-\lambda_{0}},f)&=\inf _{\substack{f=\varphi+\psi \\ \varphi \in \widetilde{L M}_{X, q_{0}}^{\lambda_{0}}, \psi \in \widetilde{L M}_{X, q_{1}}^{\lambda_{1}}}}\left(\|\varphi\|_{\widetilde{L M}_{X, q_{0}}^{\lambda_{0}}}+t^{\lambda_{1}-\lambda_{0}}\|\psi\|_{\widetilde{L M}_{X, q_{1}}^{\lambda_{1}}}\right)\\
	&\leq \|\varphi\|_{\widetilde{L M}_{X, q_{0}}^{\lambda_{0}}}+t^{\lambda_{1}-\lambda_{0}}\|\psi\|_{\widetilde{L M}_{X, q_{1}}^{\lambda_{1}}}.
	\end{aligned}
	$$
	\par 
	1.2.2. Let us estimate $  \left\|\psi_{t}\right\|_{\widetilde{L M}_{X, q_{1}}^{\lambda_{1}}} $. Since $ \left|\psi_{t}(x)\right| \leq|f(x)| $, 
	$$
	\begin{aligned}
	\left\|\psi_{t}\right\|_{\widetilde{L M}_{X, q_{1}}^{\lambda_{1}}}&=\left(\int_{0}^{\infty}\left(s^{-\lambda_{1}}\left\|\psi_{t} \chi_{B(0, s)}\right\|_{\tilde{L}_{X}}\right)^{q_{1}} \frac{d s}{s}\right)^{\frac{1}{q_{1}}} \\
	&\lesssim \left(\left(\int_{0}^{t}\left(s^{-\lambda_{1}}\left\|\psi_{t}\chi_{B(0, s)}\right\|_{\tilde{L}_{X}}\right)^{q_{1}} \frac{d s}{s}\right)^{\frac{1}{q_{1}}}+\left(\int_{t}^{\infty}\left(s^{-\lambda_{1}}\|f\chi_{B(0, s)}\|_{\tilde{L}_{X}}\right)^{q_{1}} \frac{d s}{s}\right)^{\frac{1}{q_{1}}}\right) .
	\end{aligned}
	$$
	Since \eqref{222}
	$$
	\left\|\psi_{t}\chi_{B(0, r)}\right\|_{\widetilde{X}}=s^{n} \sup _{\rho \geq s} \rho^{-n} \|  f \chi_{ B(0, \rho)}\|_{X}=\left( \frac{s}{t}\right)^n\|f \chi_{ B(0, t)}\|_{X} ~~~~~~0<s<t.
	$$
 Therefore, \\
	$$\begin{aligned}
	\left(\int_{0}^{t}\left(s^{-\lambda_{1}}\left\|\psi_{t} \chi_{B(0, s)}\right\|_{\widetilde{X}}\right)^{q_{1}} \frac{d s}{s}\right)^{\frac{1}{q_{1}}} &\leq t^{-n}\|f \chi_{ B(0, t)}\|_{\widetilde{X}}\left(\int_{0}^{t} s^{\left(n-\lambda_{1}\right) q_{1}} \frac{d s}{s}\right)^{\frac{1}{q_{1}}}\\
	&=\left(\left(n-\lambda_{1}\right) q_{1}\right)^{-\frac{1}{q_{1}}} t^{-\lambda_{1}}\|f \chi_{ B(0, t)}\|_{\widetilde{X}}\\
	&=C_{1} t^{\lambda_{0}-\lambda_{1}-n}\|f \chi_{ B(0, t)}\|_{\widetilde{X}}\left(\int_{0}^{t} s^{\left(n -\lambda_{0}\right) q_{0}} \frac{d s}{s}\right)^{\frac{1}{q_{0}}}\\
	&=C_{1} t^{\lambda_{0}-\lambda_{1}}\left(\int_{0}^{t}\left(s^{-\lambda_{0}} s^{n} \mathop{sup}\limits_{\rho \geq t} \| \rho^{n } f\chi_{ B(0, \rho)} \|_{X }\right)^{q_{0}} \frac{d s}{s}\right)^{\frac{1}{q_{0}}}\\
	&\leq C_{1} t^{\lambda_{0}-\lambda_{1}}\left(\int_{0}^{t}\left(s^{-\lambda_{0}}\|f\chi_{ B(0, s)}\|_{\widetilde{X}}\right)^{q_{0}} \frac{d s}{s}\right)^{\frac{1}{q_{0}}},\\
	\end{aligned}$$	
	$
	\text { where } C_{1}=\left(\left(n-\lambda_{1}\right) q_{1}\right)^{-1 / q_{1}}\left(\left(n-\lambda_{0}\right) q_{0}\right)^{1 / q_{0}} \text {. }
	$\par 
	Thus,
	
	$$
	\begin{aligned}
	\left\|\psi_{t}\right\|_{\widetilde{L M}_{X, q_{1}}^{\lambda}} 
&\lesssim \left( t^{\lambda_{0}-\lambda_{1}}\left(\int_{0}^{t}\left(s^{-\lambda_{0}}\|f\|_{\widetilde{X}(B(0, s))}\right)^{q_{0}} \frac{d s}{s}\right)^{\frac{1}{q_{0}}}+\left(\int_{t}^{\infty}\left(s^{-\lambda_{1}}\|f\|_{\widetilde{X}(B(0, s))}\right)^{q_{1}} \frac{d s}{s}\right)^{\frac{1}{q_{1}}}\right)\\
&\lesssim C_{2}\left(I_{1}(t)+I_{2}(t)\right),
	\end{aligned}
	$$
	where  $C_{2}=\max \left\{C_{1}, 1\right\}$ 
	\par 
	1.2.3. Let us estimate 
	$ \left\|\varphi_{t}\right\|_{\widetilde{L M}_{X, q_{0}}^{\lambda_{0}}} $. Since \eqref{222},
	$$
	\begin{aligned}
	\left\|\varphi_{t}\right\|_{\widetilde{L M}_{X, q_{0}}^{\lambda_{0}}}&=\left(\int_{0}^{\infty}\left(s^{-\lambda_{0}}\left\|\varphi_{t} \chi_{ B(0, s)}\right\|_{\widetilde{X}}\right)^{q_{0}} \frac{d s}{s}\right)^{\frac{1}{q_{0}}}	\quad~~~ \quad~~~\quad~~~ \quad	~~~ \quad\quad~~~ \quad~~~\quad~~~ \quad\\ 
				\end{aligned}
	$$
	$$
	\begin{aligned}
	\quad~~~ \quad~~~\quad~~~ \quad
	&\lesssim \left(\left(\int_{0}^{t}\left(s^{-\lambda_{0}}\|f\chi_{ B(0, s)}\|_{\widetilde{X}}\right)^{q_{0}} \frac{d s}{s}\right)^{\frac{1}{q_{0}}}+\left(\int_{t}^{\infty}\left(s^{-\lambda_{0}}\left\|\varphi_{t}\chi_{ B(0, s)}\right\|_{\widetilde{X}}\right)^{q_{0}} \frac{d s}{s}\right)^{\frac{1}{q_{0}}}\right) .
	\end{aligned} 
	$$
and
	$$
\left\|\psi_{t}\chi_{ B(0, s)}\right\|_{\tilde{X}}=s^{n} \sup _{\rho \geq s} \rho^{-n} \|  f \chi_{ B(0,  \rho)}\|_{X}=\|f\chi_{ B(0, t)}\|_{X }~~~~s>t.
$$
Therefore,
	$$
	\begin{aligned}
	\left(\int_{t}^{\infty}\left(s^{-\lambda_{0}}\left\|\varphi_{t} \chi_{ B(0, s)}\right\|_{\tilde{L}_{X}}\right)^{q_{0}} \frac{d s}{s}\right)^{\frac{1}{q_{0}}}&=\left(\lambda_{0} q_{0}\right)^{-\frac{1}{q_{0}}} t^{-\lambda_{0}}\|f\chi_{ B(0,  t)}\|_{X } \\
	&=C_{3}\left(\int_{0}^{t}\left(s^{n-\lambda_{0}} t^{-n}\|f\chi_{ B(0, t)}\|_{X }\right)^{q_{0}} \frac{d s}{s}\right)^{\frac{1}{q_{0}}} \\
	&\leq C_{3}\left(\int_{0}^{t}\left(s^{-\lambda_{0}} s^{n} \mathop{sup}\limits_{\rho \geq s} \rho^{-n}\|f\chi_{ B(0, \rho)}\|_{X }\right)^{q_{0}} \frac{d s}{s}\right)^{\frac{1}{q_{0}}}\\
	&=C_{3}\left(\int_{0}^{t}\left(s^{-\lambda_{0}}\|f \chi_{ B(0, s)}\|_{\tilde{X}}\right)^{q_{0}} \frac{d s}{s}\right)^{\frac{1}{q_{0}}}:=I_{3},
	\end{aligned}
	$$
	where  $C_{3}=\left(\frac{n-\lambda_{0}}{\lambda_{0} p_0}\right)^{\frac{1}{q_{0}}} $.\par
	Hence,$$
	\left\|\varphi_{t}\right\|_{\widetilde{L M}_{X, q_{0}}^{\lambda_{0}}} \leq C_{4} I_{3}(t), ~~~~C_{4}=1+C_{3} $$.

	1.3. So,
	$$
	\widetilde{K}\left(t^{\lambda_{1}-\lambda_{0}}, f\right) \leq C_{5}\left(I_{3}(t)+t^{\lambda_{1}-\lambda_{0}} I_{2}(t)\right),
	$$
	where $ C_{5}=C_{2}+C_{4} $, and
	$$
	\begin{aligned}
	\|f\|_{\left(L M_{X, q 0}^{\lambda_{0}} L M_{X, q_{1}}^{\lambda_{1}}\right)_{\theta, q}} 
	&\lesssim \left( \left(\int_{0}^{\infty}\left(t^{-q\left(\lambda_{1}-\lambda_{0}\right)} I_{3}(t)\right)^{q} \frac{d t}{t}\right)^{\frac{1}{q}} + \left(\int_{0}^{\infty}\left(t^{(1-q)\left(\lambda_{1}-\lambda_{0}\right)} I_{2}(t)\right)^{q} \frac{d t}{t}\right)^{\frac{1}{q}}\right)
	\\
	&\lesssim \left(C_{5}\left(J_{1}+J_{2}\right)  \right).
		\end{aligned}
	$$
	Note that
	$$
	J_{1}^{q_{0}}=\left(\int_{0}^{\infty}\left(t^{-\theta\left(\lambda_{1}-\lambda_{0}\right) q_{0}} \int_{0}^{t} s^{-\lambda_{0} q_{0}-1}\|f \chi_{B(0, s)}\|_{\tilde{X}}^{q_{0}} d s\right)^{\frac{q}{q_{0}}} \frac{d t}{t}\right)^{\frac{q_{0}}{q}}
	$$$$
	J_{2}^{q_{1}}=\left(\int_{0}^{\infty}\left(t^{(1-\theta)\left(\lambda_{1}-\lambda_{0}\right) q_{1}} \int_{t}^{\infty} s^{\lambda_{1} q_{1}-1}\|f\chi_{B(0, s)}\|_{\tilde{X}}^{q_{1}} d s\right)^{\frac{q}{q_{1}}} \frac{d t}{t}\right)^{\frac{q_{1}}{q}}
$$
	1.4. Applying the Hardy inequality in the form
	$$
	\left(\int_{0}^{\infty}\left(t^{-\alpha} \int_{0}^{t} g(s) d s\right)^{\sigma} \frac{d t}{t}\right)^{\frac{1}{\sigma}} \leq \frac{1}{\alpha}\left(\int_{0}^{\infty}\left(t^{1-\alpha} g(t)\right)^{\sigma} \frac{d t}{t}\right)^{\frac{1}{\sigma}},
	$$
	where $ \sigma \geq 1$,$ \alpha>0 $, and $ g $ is a nonnegative measurable function on  $(0, \infty) $,
	$$
	J_{1}^{q_{0}} \leq\left(\theta\left(\lambda_{1}-\lambda_{0}\right) q_{0}\right)^{-1}\left(\int_{0}^{\infty}\left(t^{-\lambda}\|f\chi_{ B(0, s)}\|_{\widetilde{X}}\right)^{q} \frac{d t}{t}\right)^{\frac{q_{0}}{q}}
	$$
	which implies, according to Lemma 4.1, that
	$$
	J_{1} \leq\left(\theta\left(\lambda_{1}-\lambda_{0}\right) q_{0}\right)^{-\frac{1}{q_{0}}}\|f\|_{\widetilde{L M}_{X, q}^{\lambda}} \leq\left(\theta\left(\lambda_{1}-\lambda_{0}\right) q_{0}\right)^{-\frac{1}{q_{0}}}\left(\frac{n p}{n-\lambda p}\right)^{\frac{1}{q}}\|f\|_{L M^{\lambda}_{X, q}}
	$$
	Similar considerations yield
	$$
	J_{2} \leq\left((1-\theta)\left(\lambda_{1}-\lambda_{0}\right) q_{1}\right)^{-\frac{1}{q_{1}}}\left(\frac{n p}{n-\lambda p}\right)^{\frac{1}{q}}\|f\|_{L M_{X, q}^{\lambda}}
	$$
	Thus, the theorem is proved under the additional assumptions $ \lambda_{0}<\lambda_{1}$  and  $q<\infty $.
	\par 
	 2. If  $q=\infty$,  the proof follows the same lines in which the integrals should be replaced by the corresponding upper bounds.	
	
	If  $\lambda_{1}<\lambda_{0}$ , then one should take into account that$$
	\|f\|_{\left(L M_{X, q_{0}}^{\lambda_{0}}, L M_{X, \tilde{q}_{1}}^{\lambda_{1}}\right)_{\theta, q}}=\left(\int_{0}^{\infty}\left(t^{-\theta} \inf _{\substack{f=\varphi+\psi \\ \varphi \in L M_{X, q_{0}}^{\lambda}, \psi \in L M_{X, q_{1}}^{\lambda}}}\left(\|\varphi\|_{L M_{X, q_{0}}^{\lambda_{0}}}+t\|\psi\|_{L M_{X, q_{1}}^{\lambda_{1}}}\right)\right)^{q} \frac{d t}{t}\right)^{\frac{1}{q}}.$$ 
	According to what has been proved above, this expression is equivalent (with constants depending only on the parameters  $q_{0}$,$ q_{1}$, $X$, $\lambda_{0}$, $\lambda_{1} $, and $ q $) to the quasinorm  $\|f\|_{L M_{X, q}^{\mu}}$, where
	$$
	\mu=(1-(1-\theta)) \lambda_{1}+(1-\theta) \lambda_{0}=\lambda,
	$$
	to the quasinorm  $\|f\|_{L M_{X, q}^{\lambda}}$. \end{proof}
\section{Vector valued maximal inequalities}
\begin{Assumption} 
	Let $ 1<u \leq \infty $, for every sequence  $\left\{f_{j}\right\}_{j=1}^{\infty} \subset L^{loc}_1\left(\mathbb{R}^{n}\right) $,  $X \in \mathbb{M}$. If
	$$
	\left\|\left(\sum_{j=1}^{\infty}\left[M f_{j}\right]^{u}\right)^{\frac{1}{u}}\right\|_{X} \lesssim\left\|\left(\sum_{j=1}^{\infty}\left|f_{j}\right|^{u}\right)^{\frac{1}{u}}\right\|_{X}.
	$$
	Especially $u = \infty$,
$$
	\left\|M\left[
	\mathop{sup}\limits_{j \in \mathbb{N}}|f_j|\right]\right\|_{X} \lesssim\left\|\mathop{sup}\limits_{j \in \mathbb{N}}|f_j|\right\|_{X}.
$$
	Then $X \in \mathbb{M}'$.
\end{Assumption}
\begin{theorem}
	Let $ 1<q \leq \infty,0\leq \lambda < \sigma,  1<v<\infty, X \in \mathbb{M}'$ . Then 
	\begin{equation}\label{3}
	\left\|\left(\sum_{j=1}^{\infty}\left(M f_{j}\right)^{v}\right)^{\frac{1}{v}}\right\|_{L M_{X, q}^{\lambda}} \lesssim\left\|\left(\sum_{j=1}^{\infty}\left|f_{j}\right|^{v}\right)^{\frac{1}{v}}\right\|_{L M_{X, q}^{\lambda}}.
	\end{equation}
	In particular,
	\begin{equation}\label{4}
	\left\|M\left[\sup _{j \in \mathbb{N}}\left|f_{j}\right|\right]\right\|_{L M_{X, q}^{\lambda}} \lesssim \left\|\sup _{j \in \mathbb{N}}\left|f_{j}\right|\right\|_{L M_{X, q}^{\lambda}}.
	\end{equation}
\end{theorem}
\begin{proof}
	In Section 3, we prove that the Hardy-Littlewood maximal operator  the operator $ M $ is bounded from $ L M_{X, q}^\lambda $ to $ L M_{X, q}^\lambda $.
	Refer to Section 3 to obtain \eqref{3}, 
we only suppose  $\theta<\infty $; the case $ \theta=\infty$  can be dealt similarly.
	$$
	\left\|\chi_{B(0, r)}\left(\sum_{j=1}^{\infty} |M f_{j}|^{v}\right)^{\frac{1}{v}}\right\|_{X} \lesssim r^{\sigma} \int_{2 r}^{\infty} t^{-\sigma-1}\left\|\chi_{B(0, t)}\left(\sum_{j=1}^{\infty}\left|f_{j}\right|^{v}\right)^{\frac{1}{v}}\right\|_{X} d t .
	$$
	Referring to the method of Section 3,  with the help of the boundedness of  $H^{*} $, we obtain the desired result.
\end{proof}	

\section{Predual spaces of the local Morrey space associated with ball quasi-Banach function spaces}
Let$ 1<q \leq \infty $.  If $supp(A)\subset B(R)$ and $\| A \|_{X} \leq R^{-\lambda-\frac{1}{q}}$, which we call the function A is a $(X , R)$-block. 
The local block space $LH_{X',q'}^{\lambda}  (\mathbb{R}^n)$ is the set of all measurable functions g. There exists a decomposition $$g(x)= \sum_{j=-\infty}^{\infty} \lambda_j A_j (x).$$
Where each of A is a $(X , 2^j)$-block   and $\{ \lambda _{j}  \}^{\infty}_ {j=-\infty} \in  l^{q '}$ and the convergence of almost all $x\in (\mathbb{R}^n)$, the norm of g is given by:
$$
\|g\| _{LH_{X',q'}^{\lambda} } := \text{inf} \left(\sum_{j=-\infty}^{\infty} |\lambda_j|^{q'}\right) ^{\frac{1}{q '}},
$$
where $\{ \lambda _{j}  \}^{\infty}_ {j=-\infty}$ covers all the above admissible expressions.
\begin{theorem}
	Let  $1<q \leq \infty $ and $0\leq \lambda < \sigma$,  then $L M_{X, q}^{\lambda} (\mathbb{R}^n)$ is the dual of $ LH_{X',q'}^{\lambda}   (\mathbb{R}^n)$ in the following sense:
	\begin{itemize}
		\item[(\romannumeral1)]
		Let $ f \in L M_{X, q}^{\lambda}  (\mathbb{R}^n)$, for any  $g \in LH_{X',q'}^{\lambda}   (\mathbb{R}^n)$, then $f g \in   L^{1}\left(\mathbb{R}^{n}\right) $, the mapping
		$$
		g \in LH_{X',q'}^{\lambda}  (\mathbb{R}^n)  \mapsto \int_{\mathbb{R}^{n}} f(x) g(x) d x \in \mathbb{C}
		$$
		may	define a continuous linear functional $ L_{f} $ on $ LH_{X',q'}^{\lambda}   $.
		\item[(\romannumeral2)]
		Conversely, any continuous linear functional  L  on $ LH_{X',q'}^{\lambda}  (\mathbb{R}^n) $ can be realized as $ L=L_{f}\left(\mathbb{R}^{n}\right) $ with a certain  $f \in L M_{X, q}^{\lambda} (\mathbb{R}^n)$. 
	\end{itemize}
	Futhermore, for all $ f \in  L M_{X, q}^{\lambda}\left(\mathbb{R}^{n}\right) $ the operator norm of $L_f$ is equivalent to $\|f\|_{ L M_{X, q}^{\lambda}}$, scilicet there exists a constant $ C>0$  such that
	$$
	C^{-1}\|f\|_{ L M_{X, q}^{\lambda}} \leq\left\|L_{f}\right\|_{LH_{X',q'}^{\lambda}   \rightarrow \mathbb{C}} \leq C\|f\|_{ L M_{X, q}^{\lambda}.}
	$$
	
\end{theorem} 
\begin{proof}
	(1)
	Let  g  be such that
	$$g:=\sum_{j=-\infty}^{\infty} \lambda_{j} A_{j},$$
	where each $ A_{j}$  is a  $\left(X^{\prime}, 2^{j}\right) $-block and $ \left\{\lambda_{j}\right\}_{j=1}^{\infty} \in l^{q^{\prime}}$  satisfies
	$$
	\left(\sum_{j=-\infty}^{\infty}\left|\lambda_{j}\right|^{q^{\prime}}\right)^{\frac{1}{q^{\prime}}} \leq 2\|g\|_{LH_{X',q'}^{\lambda} } .$$\\
	Then
	$$
	\begin{aligned}
	\|f   g\|_{L_{1}} & \leq \sum_{j=-\infty}^{\infty}\left|\lambda_{j}\right| \int_{B\left(0,2^{j}\right)}\left|f(x) A_{j}(x)\right| d x \\
	&\leq \sum_{j=-\infty}^{\infty}\left|\lambda_{j}\right|
	\|f \chi_{ B\left(0,2^{j}\right)}\|_{X} \|A_j (x)\chi_{ B\left(0,2^{j}\right)}\|_{X^\prime}\\
	&\leq \sum_{j=-\infty}^{\infty}\left|\lambda_{j}\right| 2^{-j\lambda -\frac{j}{q}}
	\|f\chi_{ B\left(0,2^{j}\right)}\|_{X}\\
	&\leq 
	\left(\sum_{j=-\infty}^{\infty}\left|\lambda_{j}\right| ^{q'}\right) ^{\frac{1}{q ^\prime}}
	\left(\sum_{j=-\infty}^{\infty}\left(2^{-j\lambda -\frac{j}{q}} 
	\|f \chi_{ B\left(0,2^{j}\right)}\|_{X} \right)^{q}\right) ^{\frac{1}{q}}\\	& \lesssim \|f\|_{L M_{X, q}^{\lambda}} \|g\|_{LH_{X',q'}^{\lambda}}.
	\end{aligned}
	$$
	(2) 
	Let L be a bounded linear functional on $ LH_{\vec{p}^\prime \theta^\prime  ,w}(\R^n)$, since the mapping
	$$
	g \in X^{\prime}\left(\mathbb{R}^{n}\right) \mapsto L\left(g \chi_{B\left(0,2^{j}\right)}\right) \in \mathbb{C}
	$$
	is a bounded linear functional, we see that L is realized by an $L^{loc}_{1}
	(\R^n)$-function
	$f$ satisfy
	\begin{equation}\label{10}
	L(g \chi_{B(0,2^j)}) = \int _{B(0,2^j)} g(x)f(x)dx
	\end{equation}
	for all $ g \in X^{\prime}\left(\mathbb{R}^{n}\right) $ and $ j \in \mathbb{Z} $. We have to check  $f \in L M_{X, q}^{\lambda}\left(\mathbb{R}^{n}\right) $, or equivalently (Proposition 2.4),
	$$
	\left(\sum_{j=-\infty}^{\infty}\left(2^{-j\lambda -\frac{j}{q}} 
	\|f \chi_{ B\left(0,2^{j}\right)}\|_{X} \right)^{q}\right) ^{\frac{1}{q}}
	$$
	To this end, choose a nonnegative $ \ell^{q^{\prime}}$ -sequence $ \left\{\rho_{j}\right\}_{j=-\infty}^{\infty} $ arbitrarily so that $ \rho_{j}=0 $ with  $|j| \gg 1$  and estimate
	$$
	\sum_{j=-\infty}^{\infty} 2^{-j\lambda -\frac{j}{q}} \rho_{j}
	\|f \chi_{ B\left(0,2^{j}\right)}\|_{X}
	$$
	Let us set
	$$
	g_{j}(x) :=\left\{\begin{array}{ll}
	P(x)~~~ ( \text{satisfy} ~\|P\|_{X'}=1) , & \text { if }	\|f \chi_{ B\left(0,2^{j}\right)}\|_{X}>0 \\
	0, & \text { otherwise. }
	\end{array}\right.
	$$
	Then each $ g_{j}$  is a $ \left(X^{\prime}, R\right) $-block
	$$
	g := \sum_{j=-\infty}^{\infty} \rho_{j} g_{j}\in LH_{X',q'}^{\lambda}  \left(\mathbb{R}^{n}\right)
	$$
	   and satisfies
	$$
	\int_{\mathbb{R}^{n}}|f(x) g(x)| d x=\sum_{j=-\infty}^{\infty} w\left(2^{j}\right) \rho_{j}	\|f \chi_{ B\left(0,2^{j}\right)}\|_{X}
	$$
	Therefore, by letting $ h(x) := \overline{\operatorname{sgn}(f(x))} g(x) $ for $ x \in \mathbb{R}^{n} $, since  $\operatorname{supp}(h) \subset B\left(0,2^{j}\right) $ for some large  j,
	$
	\int_{\mathbb{R}^{n}}|f(x) g(x)| d x=L(h)
	$
	thanks to \eqref{10} and the fact that $ \rho_{j}=0  $ if $ |j| \gg 1$ . Thus
	$$
	\begin{aligned}
	\sum_{j=-\infty}^{\infty} 2^{-j\lambda -\frac{j}{q}} \rho_{j}	\|f \chi_{ B\left(0,2^{j}\right)}\|_{X}&=\int_{\mathbb{R}^{n}}|f(x) g(x)| d x\\
	&=L(h) \leq\|L\|_{LH_{X',q'}^{\lambda}   \rightarrow \mathbb{C}}\| h\|_{LH_{X',q'}^{\lambda}}
	\\&\leq\|L\|_{LH_{X',q'}^{\lambda}   \rightarrow \mathbb{C}}\left(\sum_{j=-\infty}^{\infty}\left|\rho_{j}\right|^{q^{\prime}}\right)^{\frac{1}{q^{\prime}}} .
	\end{aligned}
	$$
\end{proof}

\section{Charcterization of the Hardy local Morrey-type spaces associated with ball quasi-Banach function spaces in terms of the grand maxiaml operators and the heat kernel}
We characterize the space $ L M_{X, q}^{\lambda}\left(\mathbb{R}^{n}\right)$ in terms of the heat kernel.
Let us show that $ L M_{X, q}^{\lambda}\left(\mathbb{R}^{n}\right)$ and $H  L M_{X, q}^{\lambda}\left(\mathbb{R}^{n}\right)$ are isomorphic.
\begin{theorem}
	Let $ 1<q\leq \infty$, $0\leq \lambda < \sigma$, and $X\in \mathbb{M}'$.
	\begin{itemize}
		\item[(\romannumeral1)]
		If $f \in  L M_{X, q}^{\lambda}\left(\mathbb{R}^{n}\right)$, then $f \in H  L M_{X, q}^{\lambda}\left(\mathbb{R}^{n}\right)$.
		\item[(\romannumeral2)]
		If $f \in H  L M_{X, q}^{\lambda}\left(\mathbb{R}^{n}\right)$, then $f$ is represented by a measurable function $g\in L M_{X, q}^{\lambda}\left(\mathbb{R}^{n}\right)$.
	\end{itemize}
	If $f \in  L M_{X, q}^{\lambda}\left(\mathbb{R}^{n}\right)$, then
	\begin{equation}\label{11}
	\|f\|_{ L M_{X, q}^{\lambda}} \leq\|f\|_{H  L M_{X, q}^{\lambda}} \leq C\|f\|_{ L M_{X, q}^{\lambda}}
	\end{equation}
\end{theorem}
\begin{proof}
	(1)
	Proposition 2.2 has proved that $ L M_{X, q}^{\lambda}\left(\mathbb{R}^{n}\right) \hookrightarrow \mathcal{S}^{\prime}\left(\mathbb{R}^{n}\right)$. Also\cite{GH2017}
	$
	\sup _{t>0}\left|e^{t \Delta} f\right| \leq M f .
	$
	\par 
	In Section 3, the $ L M_{X, q}^{\lambda}\left(\mathbb{R}^{n}\right)$-boundedness of the Hardy-Littlewood maximal operator, we see that $f \in H  L M_{X, q}^{\lambda}\left(\mathbb{R}^{n}\right)$ and that the right inequality in \eqref{11} follows.\par
	(2) Due to Theorem 6.1, the dual of $L H_{X^{\prime},q^{\prime}}^{\lambda} \left(\mathbb{R}^{n}\right)$ is isomorphic to $ L M_{X, q}^{\lambda}\left(\mathbb{R}^{n}\right)$. Let $L: h \in  L M_{X, q}^{\lambda}\left(\mathbb{R}^{n}\right) \mapsto L_{h} \in\left(L H_{X', q'}^{\lambda}\left(\mathbb{R}^{n}\right)\right)^{*}$ be an isomorphism in Theorem 6.1. By the Banach-Alaoglu theorem, there exists a positive decreasing sequence $\left\{t_{j}\right\}_{j=1}^{\infty} \subset(0,1)$ such that $L_{e^{t} \Delta_{f}}$ is convergent to $G=L_{g} \in$ $\left(L H_{X', q'}^{\lambda}\left(\mathbb{R}^{n}\right)\right)^{*}$ for some $g \in  L M_{X, q}^{\lambda}\left(\mathbb{R}^{n}\right)$ in the weak-* sense. Observe that
	$$
	\begin{aligned}
	\|g\|_{L M_{X, q}^{\lambda}} &\sim\left\|L_{g}\right\|_{\left(L H_{X', q'}^{\lambda}\right)^{*}}\\
	&\leq \liminf _{j \rightarrow \infty}\left\|L_{e^{t} \Delta^{\Delta}}\right\|_{\left(L H_{X', q'}^{\lambda}\right)^{*}}\\
	&\sim \liminf _{j \rightarrow \infty}\left\|e^{t_{j} \Delta} f\right\|_{ L M_{X, q}^{\lambda}} \leq\|f\|_{H  L M_{X, q}^{\lambda}} .
	\end{aligned}
	$$
	Meanwhile, since $f \in \mathcal{S}^{\prime}\left(\mathbb{R}^{n}\right), e^{t_{j} \Delta} f$ is convergent to $f \in \mathcal{S}^{\prime}\left(\mathbb{R}^{n}\right)$. Thus, we conclude $\mathcal{S}^{\prime}\left(\mathbb{R}^{n}\right) \ni f=g \in  L M_{X, q}^{\lambda}\left(\mathbb{R}^{n}\right)$.
	The left inequality in \eqref{11} follows since the spaces $ L M_{X, q}^{\lambda}\left(\mathbb{R}^{n}\right)$ is isomorphic to the dual of $L H_{X', q'}^{\lambda}\left(\mathbb{R}^{n}\right)$. Thus, from Lebesgue's differentiation theorem,
	$$
	\|f\|_{ L M_{X, q}^{\lambda}} \leq\left\|\sup _{t>0}\left|e^{t \Delta} f\right|\right\|_{ L M_{X, q}^{\lambda}}=\|f\|_{H  L M_{X, q}^{\lambda}}.
	$$
\end{proof}
In terms of the grand maximal opetator in Definition 2.3, can rephrase Theorem $7.1$ as follows:

\begin{theorem}
	Let $ 1<q\leq \infty$, $0\leq \lambda < \sigma$, and $X\in \mathbb{M}'$.
	\begin{itemize}
		\item[(\romannumeral1)]
		If $f \in  L M_{X, q}^{\lambda}\left(\mathbb{R}^{n}\right)$, then $\mathcal{M} f \in  L M_{X, q}^{\lambda}\left(\mathbb{R}^{n}\right)$.
		\item[(\romannumeral2)]
		Let $f \in \mathcal{S}^{\prime}\left(\mathbb{R}^{n}\right)$, if $\mathcal{M} f \in  L M_{X, q}^{\lambda}\left(\mathbb{R}^{n}\right)$, then $f$ is represented by a measurable function $g\in  L M_{X, q}^{\lambda}\left(\mathbb{R}^{n}\right)$.
	\end{itemize}
	If $f \in  L M_{X, q}^{\lambda}\left(\mathbb{R}^{n}\right)$, then $C^{-1}\|f\|_{ L M_{X, q}^{\lambda}} \leq\|\mathcal{M} f\|_{ L M_{X, q}^{\lambda}} \leq C\|f\|_{L M_{X, q}^{\lambda}} .$
\end{theorem}
\begin{proof}
	The implication ($i$) $\Longrightarrow(ii)$ immediately follows from the pointwise inequality $\mathcal{M} f(x) \lesssim  M f(x)$. The converse implication ($ii$) $\Longrightarrow$ ($i$) follows from the pointwise estimatee $\left|e^{t \Delta} f(x)\right| \lesssim \mathcal{M} f(x)$. Indeed, from this pointwise estimate, we conclude $
	\mathop{sup}\limits_{t>0}\left|e^{t \Delta} f\right| \in  L M_{X, q}^{\lambda}\left(\mathbb{R}^{n}\right)$. Thus, we are in the position of applying Theorem $7.1$ to receive $f \in  L M_{X, q}^{\lambda}\left(\mathbb{R}^{n}\right)$.
\end{proof}
\begin{remark}\cite{SY2017}
	$H  L M_{X, q}^{\lambda}$ returne to the norm of Hardy-type spaces for ball quasi-Banach function spaces $H_{L M_{X, q}^{\lambda}}$.
	$$
	\|f\|_{H_{L M_{X, q}^{\lambda}}}\left(\mathbb{R}^n\right) \sim\left\|\sup _{t \in(0, \infty)} \mid e^{t \Delta} f| \right\| _{L M_{X, q}^{\lambda}}
	$$
\end{remark}
\section{Atomic decomposition of the local Morrey space associated with ball quasi-Banach function spaces}
\begin{theorem}
	Let $ 1<q\leq \infty$, $0\leq \lambda < \sigma$, $X\in \mathbb{M}'$ and
	\begin{equation}\label{12}
\sigma_1<\sigma -\lambda
	\end{equation}
	And $ \left\{Q_{j}\right\}_{j=1}^{\infty} \subset \mathcal{Q}\left(\mathbb{R}^{n}\right),\left\{a_{j}\right\}_{j=1}^{\infty} \subset X_1 \left(\mathbb{R}^{n}\right) ,  \left\{\lambda_{j}\right\}_{j=1}^{\infty} \subset[0, \infty)$  satisfying
	$$
	\left\|a_{j}\right\|_{X_1} \leq\left\|\chi_{Q_j}\right\|_{X_1}=r^{\sigma_1}, \quad \operatorname{supp}\left(a_{j}\right) \subset Q_{j}, \quad \sum_{j=1}^{\infty} \| \lambda_{j} \chi_{Q_{j}} \|_{ L M_{X, q}^{\lambda}}<\infty.
	$$
	Then the series  $f := \sum_{j=1}^{\infty} \lambda_{j} a_{j}$  converges in $ L_{\mathrm{loc}}^{1}\left(\mathbb{R}^{n}\right) $ and in the Schwartz space $ \mathcal{S}^{\prime}\left(\mathbb{R}^{n}\right) $ of tempered distributions and satisfies the estimate
	\begin{equation}\label{13}
	\|f\|_{ L M_{X, q}^{\lambda}} \lesssim \left\|\sum_{j=1}^{\infty} \lambda_{j} \chi Q_{j}\right\|_{ L M_{X, q}^{\lambda}.}
	\end{equation}
	
\end{theorem}
\begin{lemma}
	Let $ 1<q\leq \infty$, $0\leq \lambda < \sigma$, each $A_{j}$ be a $\left(X ', 2^{j}\right)$-block and $\left\{\rho_{j}\right\}_{j=-\infty}^{\infty} \in \ell^{q^{\prime}}$. Suppose $\sigma$ and $\sigma_1$ satisfies \eqref{12}. Then
	$$
	h := \sum_{j=-\infty}^{\infty} \rho_{j}M\left[ \| A_j \chi_{ B(0, r)}\| _{X'_1}\right]\in LH_{X',q'}^{\lambda}  \left(\mathbb{R}^{n}\right)
	,~~~~
	\|h\|_{LH_{X',q'}^{\lambda}  } \leq C\left(\sum_{j=-\infty}^{\infty}\left|\rho_{j}\right|^{q^{\prime}}\right)^{1 / q^{\prime}}.
	$$
\end{lemma}
\begin{proof}
	By the $X'_1$-boundedness of the Hardy-Littlewood maximal operator and $q^{\prime}<\infty$, 
	$$
\sum_{j=-\infty}^{\infty} \rho_{j}  \chi_{ B(0,  2^{j+1})} M\left[ \| A_j \chi_{ B(0, r)}\| _{X'_1}\right]\in LH_{X',q'}^{\lambda}  \left(\mathbb{R}^{n}\right)
	$$
	and
	$$
	\left\|\sum_{j=-\infty}^{\infty} \rho_{j}  \chi_{ B(0,  2^{j+1})} M\left[ \| A_j \chi_{ B(0, r)}\| _{X'_1}\right] \right\|_{LH_{X',q'}^{\lambda} } \lesssim \left(\sum_{j=-\infty}^{\infty}\left|\rho_{j}\right|^{q^{\prime}}\right)^{1 / q^{\prime}}.
	$$
	Meanwhile, combining $\| A_j \| _{X'_1} \leq (2^j)^{\sigma-\sigma_1} \| A_j\| _{X'} \leq (2^j)^{\sigma-\sigma_1-\lambda-\frac{1}{q}} $ and \eqref{12}, therefore,
	$$
	\begin{aligned}
	&\sum_{j=-\infty}^{\infty} \rho_{j} \chi_{B\left(2^{j+1}\right)^{c}}M\left[ \| A_j \chi_{ B(0, r)}\| _{X'_1}\right]\\
	&=\sum_{k=0}^{\infty} \sum_{j=-\infty}^{\infty} \rho_{j} \chi_{B\left(2^{j+k+2}\right) \backslash B\left(2^{j+k+1}\right)}M\left[ \| A_j \chi_{ B(0, r)}\| _{X'_1}\right] \\
	& \leq C \sum_{k=0}^{\infty} \sum_{j=-\infty}^{\infty} \rho_{j} \frac{1}{2^{(j+k)\sigma'_1}}\left\| A_j  \chi_{ B(0,  2^{j})}\right\|_{X'_1}
	\chi_{B\left(2^{j+k+2}\right) \backslash B\left(2^{j+k+1}\right)} \\
		\end{aligned}
	$$
	$$
	\begin{aligned}
	\quad \quad 
	&\leq C \sum_{k=0}^{\infty} \sum_{j=-\infty}^{\infty} \rho_{j} \frac{2^{j \sigma-j \sigma_1}}{2^{(j+k)\sigma'_1}}\left\| A_j  \chi_{ B(0,  2^{j})}\right\|_{X'} \chi_{B\left(2^{j+k+2}\right) \backslash B\left(2^{j+k+1}\right)^{.}} \\
	&\leq C \sum_{k=0}^{\infty} \sum_{j=-\infty}^{\infty} \rho_{j} \frac{2^{j \sigma-j \sigma_1} (2^j)^{-\lambda-\frac{1}{q}} }{2^{(j+k)\sigma'_1}} \chi_{B\left(2^{j+k+2}\right) \backslash B\left(2^{j+k+1}\right)} 
	\lesssim \sum_{j=-\infty}^{\infty} \rho_{j}.
	\end{aligned}
	$$
	
\end{proof}
Next, to prove Theorem 8.1.
\begin{proof}
	To prove \eqref{13}, we resort to the duality obtained in Theorem 6.1
	$$
	\|f\|_{ L M_{X, q}^{\lambda}} = \sup \left\{\left|\int_{\mathbb{R}^{n}} f(x) g(x) d x\right|:\|g\|_{LH_{X',q'}^{\lambda}  }=1\right\} .
	$$
	We can assume that $\left\{\lambda_{j}\right\}_{j=1}^{\infty}$ is finitely supported thanks to the monotone convergence theorem. Let us assume in addition that the $a_{j}$ are non-negative without loss of generality. 
	$$
	g := \sum_{k=-\infty}^{\infty} \rho_{k} A_{k}, \quad G := \sum_{k=-\infty}^{\infty}\left|\rho_{k}\right| M\left[ \| A_j \chi_{ B(0, r)}\| _{X'_1}\right],
	$$
	where each $A_{k}$ is a $\left(\vec{p}^{\prime}, 2^{j}\right)$-block, Lemma 8.1 and
	$$
	\sum_{k=-\infty}^{\infty}\left|\rho_{k}\right|^{q^{\prime}} \leq 1.
	$$
	Then 
	$$
	\begin{aligned}
	\left|\int_{\mathbb{R}^{n}} f(x) g(x) d x\right| & \leq \sum_{(j, k) \in \mathbb{N} \times \mathbb{Z}} \lambda_{j}\left|\rho_{k}\right| \int_{B\left(2^{k}\right) \cap Q_{j}} a_{j}(x)\left|A_{k}(x)\right| d x \\
	& \leq \sum_{(j, k) \in \mathbb{N} \times \mathbb{Z}} \lambda_{j}\left|\rho_{k}\right|  \| a_j \chi_{ B(0, r)}\| _{X_1} \| A_j \chi_{ B(0, r)}\| _{X'_1} \\
	& \lesssim \sum_{(j, k) \in \mathbb{N} \times \mathbb{Z}} \lambda_{j}\left|\rho_{k}\right|  \int_{Q_{j}} M\left[ \| A_j \chi_{ B(y, r)}\| _{X'_1}\right] d x 
	< \infty.
	\end{aligned}
	$$ 
\end{proof}

With the aid of Proposition 2.2, we extend into Theorem 8.2.\\
\begin{theorem}
	Satisfying the conditions of theorem 8.1 but where $ \left\{a_{j}\right\}_{j=1}^{\infty} \subset L^{\infty}\left(\mathbb{R}^{n}\right) $ such that $ f := \sum_{j=1}^{\infty} \lambda_{j} a_{j}$  converges in $ \mathcal{S}^{\prime}\left(\mathbb{R}^{n}\right) \cap L^{\mathrm{loc}}_{1}\left(\mathbb{R}^{n}\right) $, that
	\begin{equation}\label{78}
	\left|a_{j}\right| \leq \chi_{Q_{j}}, \quad \int_{\mathbb{R}^{n}} x^{\alpha} a_{j}(x) d x=0,
	\end{equation}
	for all multi-indices $ \alpha=\left(\alpha_{1}, \alpha_{2}, \ldots, \alpha_{n}\right) $ with  $|\alpha| := \alpha_{1}+\alpha_{2}+\cdots+\alpha_{n} < \infty $ and, that for all  $v>0$ 
	\begin{equation}\label{79}
	\left\|\left(\sum_{j=1}^{\infty}\left(\lambda_{j} \chi_{Q_{j}}\right)^{v}\right)^{1 / v}\right\|_{ L M_{X, q}^{\lambda}} \leq C_{v}\|f\|_{ L M_{X, q}^{\lambda}} .
	\end{equation}
	Here the constant  $C_{v}>0$  is independent of  f .
\end{theorem}

\begin{lemma}\cite{BS2014}
	Let $\varphi \in \mathcal{S}\left(\mathbb{R}^{n}\right)$. With the same notation as Lemma $2.1$, then
	\begin{equation}\label{30}
	\left|\left\langle b_{j}, \varphi\right\rangle\right| \leq C_{\varphi}\left\{\sum_{l=0}^{\infty}\left(\frac{1}{2^{l n}}\left\|\mathcal{M} f \cdot \chi \mathcal{O}_{j}\right\|_{L^{1}\left(B\left(2^{l}\right)\right)}\right)^{\mu}\right\}^{1 / \mu},
	\end{equation}
	where $\mu := \frac{n+d+1}{n}$ and the constant $C_{\varphi}$ in $\eqref{30}$ depends on $\varphi$ but not on $j$ or $k$.
\end{lemma}

\begin{lemma}
	Let $ 1<q\leq \infty$, $0\leq \lambda < \sigma$, $X\in \mathbb{M}'$, $f \in  L M_{X, q}^{\lambda}\left(\mathbb{R}^{n}\right)$ and  $\sigma, \sigma_1$ satisfies \eqref{12}. Then in the notation of Lemma 2.1, in the topology of $\mathcal{S}^{\prime}\left(\mathbb{R}^{n}\right)$,  $g_{j} \rightarrow 0$ as $j \rightarrow-\infty$ and $b_{j} \rightarrow 0$ as $j \rightarrow \infty$. In particular
	$$
	f=\sum_{j=-\infty}^{\infty}\left(g_{j+1}-g_{j}\right).
	$$
	
\end{lemma}
\begin{proof}
	Observe that
	$$
	\begin{aligned}
	\frac{1}{2^{l n}}\left\|\mathcal{M} f \cdot \chi \mathcal{O}_{j}\right\|_{L^{1}\left(B\left(0,2^{l}\right)\right)} &\lesssim \frac{1}{2^{l n}}\|\mathcal{M} f\|_{L^{1}\left(B\left(0,2^{l}\right)\right)}\\
	&\lesssim \frac{1}{2^{l
			\sigma}}\|\mathcal{M} f \chi_{B(0,2^l)}\|_{X}\\
	&\lesssim \frac{1}{2^{l
			\sigma} (2^{l})^{-\lambda-\frac{1}{q}}}\|f\|_{H  L M_{X, q}^{\lambda}}\\
	&\lesssim \frac{1}{2^{l \sigma
		}}\|f\|_{L M_{X, q}^{\lambda}}.
	\end{aligned}
	$$
	Note that \eqref{12} and $\mu := \frac{n+d+1}{n}$.
	$$
	\sum_{l=1}^{\infty}\left(\frac{1}{2^{\l
			\sigma} 2^{-\l\lambda-\frac{\l}{q}}}\right)^{\mu}<\infty
	$$
	Consequently, we may use the Lebesgue convergence theorem to conclude that $b_{j} \rightarrow 0$ as $j \rightarrow \infty$. Hence, it follows that $f=\lim _{j \rightarrow \infty} g_{j}$ in $\mathcal{S}^{\prime}\left(\mathbb{R}^{n}\right)$. Consequently, it follows from Lemma 2.1 that $f=\lim _{j \rightarrow \infty} g_{j}=\lim _{j, k \rightarrow \infty} \sum_{l=-k}^{j}\left(g_{l+1}-g_{l}\right)$ in $\mathcal{S}^{\prime}\left(\mathbb{R}^{n}\right)$.
\end{proof}
\par 
Next, to prove Theorem 8.2.
\begin{proof}
	For each $j \in \mathbb{Z}$, consider the level set
	$$
	\mathcal{O}_{j} :=\left\{x \in \mathbb{R}^{n}: \mathcal{M} f(x)>2^{j}\right\}
	$$
	Then it follows immediately from the definition that
	$$
	\mathcal{O}_{j+1} \subset \mathcal{O}_{j} .
	$$
	Apply Lemma 2.1, then $f$ can be decomposed as
	$$
	f=g_{j}+b_{j}, \quad b_{j}=\sum_{k} b_{j, k}, \quad b_{j, k}=\left(f-c_{j, k}\right) \eta_{j, k}
	$$
	where each $b_{j, k}$ is supported in a cube $Q_{j, k}$ as described in Lemma 2.1.
	$$
	f=\sum_{j=-\infty}^{\infty}\left(g_{j+1}-g_{j}\right)
	$$
	with the series converging in the sense of distributions from Lemma 6.4. 
	$$
	f=\sum_{j, k} A_{j, k}, \quad g_{j+1}-g_{j}=\sum_{k} A_{j, k} \quad(j \in \mathbb{Z})
	$$
	in the sense of distributions, where each $A_{j, k}$, supported in $Q_{j, k}$, satisfies the pointwise estimate $\left|A_{j, k}(x)\right| \leq C_{0} 2^{j}$ for some universal constant $C_{0}$ and the moment condition $\int_{\mathbb{R}^{n}} A_{j, k}(x) q(x) d x=0$ for every $q(x) \in \mathcal{P}_{d}\left(\mathbb{R}^{n}\right)$. With these observations in mind, write
	$$
	a_{j, k} := \frac{A_{j, k}}{C_{0} 2^{j}}, \quad \kappa_{j, k} := C_{0} 2^{j} .
	$$
	Then we shall obtain that each $a_{j, k}$ satisfies
	$$
	\left|a_{j, k}\right| \leq \chi_{Q_{j, k}}, ~~~\quad~  \int_{\mathbb{R}^{n}} x^{\alpha} a_{j, k}(x) d x=0 
	$$
	and that $f=\sum_{j, k} \kappa_{j, k} a_{j, k}$ in the topology of $H  L M_{X, q}^{\lambda}\left(\mathbb{R}^{n}\right)$. Rearrange $\left\{a_{j, k}\right\}$ to obtain $\left\{a_{j}\right\}$. Do the same rearrangement for  $\left\{\lambda_{j, k}\right\}$.
	To establish \eqref{79}, write
	$$
	\beta  := \left\|\left(\sum_{j=-\infty}^{\infty}\left|\lambda_{j} \chi_{Q_{j}}\right|^{v}\right)^{1 / v}\right\|_{L M_{X, q}^{\lambda}} .
	$$
	Since
	$$
	\left\{\left(\kappa_{j, k} ; Q_{j, k}\right)\right\}_{j, k}=\left\{\left(\lambda_{j} ; Q_{j}\right)\right\}_{j},
	$$
	we have
	$$
	\beta =\left\|\left(\sum_{j=-\infty}^{\infty} \sum_{k \in K_{j}}\left|\kappa_{j, k} \chi_{Q_{j, k}}\right|^{v}\right)^{1 / v}\right\|_{L M_{X, q}^{\lambda}} .
	$$
	By using the definition of $\kappa_{j}$, we then have
	$$
	\beta =C_{0}\left\|\left(\sum_{j=-\infty}^{\infty} \sum_{k \in K_{j}}\left|2^{j} \chi_{Q_{j, k}}\right|{ }^{v}\right)^{1 / v}\right\|_{L M_{X, q}^{\lambda}}=C_{0}\left\|\left(\sum_{j=-\infty}^{\infty} 2^{j v} \sum_{k \in K_{j}} \chi_{Q_{j, k}}\right)^{1 / v}\right\|_{L M_{X, q}^{\lambda}} .
	$$
	Observe that \eqref{88}, together with the bounded overlapping property, yields
	$$
	\chi_{\mathcal{O}_{j}}(x) \leq \sum_{k \in K_{j}} \chi_{Q_{j, k}}(x) \leq \sum_{k \in K_{j}} \chi_{200 Q_{j, k}}(x) \lesssim \chi_{\mathcal{O}_{j}}(x) \quad\left(x \in \mathbb{R}^{n}\right) .
	$$
	Thus, 
	$$
	\beta  \lesssim \left\|\left(\sum_{j=-\infty}^{\infty}\left(2^{j} \chi_{\mathcal{O}_{j}}\right)^{v}\right)^{1 / v}\right\|_{L M_{X, q}^{\lambda}} .
	$$
	Recalling that $\mathcal{O}_{j} \supset \mathcal{O}_{j+1}$ for each $j \in \mathbb{Z}$, 
	$$
	\sum_{j=-\infty}^{\infty}\left(2^{j} \chi_{\mathcal{O}_{j}}(x)\right)^{v} \sim\left(\sum_{j=-\infty}^{\infty} 2^{j} \chi_{\mathcal{O}_{j}}(x)\right)^{v} \sim\left(\sum_{j=-\infty}^{\infty} 2^{j} \chi_{\mathcal{O}_{j} \backslash \mathcal{O}_{j+1}}(x)\right)^{v} \quad\left(x \in \mathbb{R}^{n}\right) .
	$$
	Then, 
	$$
	\beta  \lesssim\left\|\sum_{j=-\infty}^{\infty} 2^{j} \chi_{\mathcal{O}_{j} \backslash \mathcal{O}_{j+1}}\right\|_{L M_{X, q}^{\lambda}} .
	$$
	It follows by the definition of $\mathcal{O}_{j}$ that $2^{j}<\mathcal{M} f(x)$ for all $x \in \mathcal{O}_{j}$. Hence, 
	$$
	\beta  \lesssim \left\|\sum_{j=-\infty}^{\infty} \chi_{\mathcal{O}_{j} \backslash \mathcal{O}_{j+1}} \mathcal{M} f\right\|_{L M_{X, q}^{\lambda}} \lesssim\|\mathcal{M} f\|_{L M_{X, q}^{\lambda}},
	$$
	So we receive the proof of Theorem 8.2.
\end{proof}
\section{The Hardy operator  on the local Morrey space associated with ball quasi-Banach function spaces}
\begin{theorem}
	Suppose $ 1<q\leq \infty$, $0\leq \lambda < \sigma$, $X\in \mathbb{M}'$. Then
	$
	\|H f\|_{ L M_{X, q}^{\lambda}} \lesssim \|f\|_{ L M_{X, q}^{\lambda}}.
	$
\end{theorem}
\begin{proof}
	Let $f=\sum_{j=1}^{\infty} \lambda_{j} a_{j}$, $\mu$ stands for the Haar measure of $\mathrm{SO}(n)$\cite{GH2017}.
	$$
	S f(x) := \int_{\mathrm{SO}(n)} f(A x) d \mu(A)
	$$
	Note that
	$$
	S:  L M_{X, q}^{\lambda}(\mathbb{R}) \rightarrow  L M_{X, q}^{\lambda}(\mathbb{R})
	$$
	is a bounded linear opeator. Since
	$$
	\begin{aligned}
	H f(x) & \sim \frac{1}{|x|^{n}} \int_{B(|x|)} f(y) d y \\
	&=\int_{\operatorname{SO}(n)} \frac{1}{|A x|^{n}} \int_{B(|A x|)} f(y) d y d \mu(A) \\
	&=\int_{\operatorname{SO}(n)} \frac{1}{|x|^{n}} \int_{B(|x|)} f(A y) d y d \mu(A) 
	=H S f(x),
	\end{aligned}
	$$
	therefore
	$$
	H f=H S f=\sum_{j=1}^{\infty} \lambda_{j} H S a_{j} .
	$$
	since $a_{j}$ is compactly supported 
	$
	\left|H S a_{j}\right| \lesssim S \chi_{Q_{j}},
	$ and Theorem 8.2, 
	$$
	\begin{aligned}
	\|H f\|_{ L M_{X, q}^{\lambda}} & \leq\left\|\sum_{j=1}^{\infty} \lambda_{j} H S a_{j}\right\|_{ L M_{X, q}^{\lambda}} 
	\lesssim \left\|\sum_{j=1}^{\infty} \lambda_{j} S_{\chi_{Q_j}}\right\|_{ L M_{X, q}^{\lambda}} \\
	& \lesssim \left\|\sum_{j=1}^{\infty} \lambda_{j} \chi_{Q_{j}}\right\|_{ L M_{X, q}^{\lambda}}
	\lesssim \|f\|_{ L M_{X, q}^{\lambda}}
	\end{aligned}
	$$
	
\end{proof}

\hspace*{-0.6cm}\textbf{\bf Competing interests}\\
	The authors declare that they have no competing interests.\\
	
\hspace*{-0.6cm}\textbf{\bf Funding}\\
	The research was supported by Natural Science Foundation of China (Grant Nos. 12061069).\\
	
\hspace*{-0.6cm}\textbf{\bf Authors contributions}\\
	All authors contributed equality and significantly in writing this paper. All authors read and approved the final manuscript.\\
	
\hspace*{-0.6cm}\textbf{\bf Acknowledgments}\\
	The authors would like to express their thanks to the referees for valuable advice regarding previous version of this paper.\\
	
	\hspace*{-0.6cm}\textbf{\bf Authors detaials}\\
	Mingwei shi and Jiang Zhou*, moluxiangfeng888@163.com and zhoujiang@xju.edu.cn, College of Mathematics and System Science, Xinjiang University, Urumqi, 830046, P.R China.

\bigskip
\noindent Mingwei shi and Jiang Zhou\\
\medskip
\noindent
College of Mathematics and System Sciences\\
Xinjiang University\\
Urumqi 830046\\
\smallskip
\noindent{D-mail }:\\
\texttt{moluxiangfeng888@163.com} (Mingwei shi)\\
\texttt{zhoujiang@xju.edu.cn} (Jiang Zhou)
\bigskip \medskip
\end{document}